\documentclass[final,onefignum,onetabnum]{siamart251216}
\usepackage{amsfonts}
\usepackage{graphicx}
\usepackage{epstopdf}
\usepackage{algorithmic}
\usepackage{amssymb}
\usepackage{bm}
\usepackage{booktabs}
\usepackage{amsmath, amssymb, mathtools}
\usepackage{xcolor}
\usepackage{cite}

\newcommand{\di}{\,\mathrm{d}}

\newcommand{\DKL}{D_{\mathrm{KL}}}
\newcommand{\maybeincludegraphics}[2][]{%
  \IfFileExists{#2}{\includegraphics[#1]{#2}}{%
    \fbox{\parbox{0.42\textwidth}{\centering\textcolor{red}{Missing figure: #2}}}%
  }%
}

\newsiamremark{example}{Example}
\newsiamthm{assumption}{Assumption}
\newsiamremark{remark}{Remark}

\headers{Diffusion Models Analysis via Entropy Production}{H. Wu and Z. Zhang}

\title{
A Unified Kullback--Leibler Divergence Analysis of Generative Diffusion Models via  Entropy Production Rate
\thanks{Z. Zhang was partially supported by the National Natural Science Foundation of China (Projects 92470103), the Hong Kong RGC Grant (Projects 17304324 and 17300325), the Seed Funding Programme for Basic Research (HKU), and the Hong Kong RGC Research Fellow Scheme 2025. The computations were performed at the Information Technology Services, The University of Hong Kong.}}

\author{Han Wu\thanks{Department of Mathematics, The University of Hong Kong, Pokfulam Road, Hong Kong SAR, u3009967@connect.hku.hk.
}\and Zhiwen Zhang\thanks{Corresponding author. Department of Mathematics, The University of Hong Kong, Pokfulam Road, Hong Kong SAR, P.R.China. Materials Innovation Institute for Life Sciences and Energy (MILES), HKU-SIRI, Shenzhen, 518045, P.R. China. Zhangzw@hku.hk.
}}

\begin{document}

\maketitle

\begin{abstract}
We introduce a unified framework for the error analysis of generative models based on the entropy production rate of the forward-reverse diffusion process pair. For a pair of continuity equation flows, the rate admits a closed velocity form identity whose time integral decomposes the terminal Kullback--Leibler (KL) divergence into the sum of an initialization error, a score approximation error, and a time-discretization error. By analyzing the entropy production at the level of marginal distributions, rather than in path space, our framework yields a sharp convergence rate of $\mathcal{O}(h^2)$
 for the Euler-Maruyama sampler, where $h$
 is the step size. This improves upon the $\mathcal{O}(h)$ rates typically obtained from Girsanov's path-space analyses. Furthermore, our framework unifies the analysis of score-based SDEs, probability-flow ODEs, and stochastic interpolants by varying diffusion coefficients within a single inequality, revealing the trade-off between deterministic and stochastic sampling. Numerical experiments confirm the predicted scaling with step size and terminal time.

\end{abstract}

\begin{keywords}
  Diffusion models,generative model, entropy production rate, score-based generative models, flow matching, stochastic interpolants, convergence analysis, Kullback--Leibler (KL) divergence, stochastic differential equations (SDEs).
\end{keywords}

\begin{AMS}
  60J60, 65C30, 68T07, 94A17.
\end{AMS}

\section{Introduction}\label{sec:intro}

Generative models rest on a duality between two stochastic processes
\cite{sohl2015deep,ho2020denoising,song2021score}.
A forward process gradually turns the data distribution into a tractable one,
and generation requires a
backward process that turns the base distribution back into the target. The central problem is
that these two processes must be matched: the marginal
distributions visited by the backward process must coincide with the
time reversal of the forward marginals. Solving this generative
matching problem involves two fundamental components:
a deterministic flow that transports probability mass, and stochastic fluctuations
that supply the randomness of individual samples. The stochastic differential equation (SDE) provides a natural framework for this task \cite{song2021score} because it explicitly separates
the dynamics into deterministic transport (the drift part $b$) and stochastic fluctuation
(the diffusion part $A$). As will be shown in section \ref{sec:unified-comparison}, deterministic methods like probability-flow ODE \cite{song2021score,karras2022elucidating}, flow matching \cite{lipman2023flow,liu2023rectified}, and stochastic interpolants \cite{albergo2023stochastic} are
embedded naturally as diffusion processes whose diffusion coefficient is reduced or vanishes.

Another key mathematical advantage of this formulation is its structural integrity under time reversal. By Anderson's theorem \cite{anderson1982reverse}, later extended in
considerable generality \cite{haussmann1986time,millet1989integration,follmer1986time,cattiaux2022time},
the time reversal of a diffusion process is identically another diffusion process.
The forward and backward processes therefore live in the exact same algebraic
class, which yields a symmetric framework for bounding their divergence.
While deterministic interpolants are highly effective in practice
\cite{lipman2023flow,liu2023rectified,albergo2023stochastic}, their
time-reversals do not generally share the same simple Markovian property without
additional structural assumptions.

Furthermore, the SDE's explicit separation of transport and fluctuation maps
intuitively onto the physical requirements of generation where transport
directs the mass flow and diffusion dictates how the process forgets its
initialization and self corrects. A natural language for this intuition is
Nelson's stochastic mechanics \cite{nelson1966derivation,nelson1967dynamical,chavanis2024connection}. For a diffusion process, Nelson decomposed the
dynamics into a current velocity  $v_c$, the time symmetric part that transports
probability mass and obeys a deterministic continuity equation, and an
osmotic velocity, the time antisymmetric part $A\nabla\log\rho$ produced
by diffusion. This same symmetric/antisymmetric splitting underlies the
thermodynamic analysis of general diffusions \cite{qian2002thermodynamics,qian2015thermodynamics,jiang2004general}.
Generation exploits exactly this structure: the forward process
fixes the marginals and hence the current velocity, while running it backward
requires supplying the osmotic correction, which is exactly the score
$\nabla\log\rho$. The flow matching \cite{lipman2023flow,liu2023rectified,albergo2023stochastic} is exactly to learn the current velocity. And score-based generation \cite{song2019generative,vincent2011connection,song2021score} can also be viewed as the program of
learning this osmotic velocity.

While extensive analyses have been conducted within the SDE framework
\cite{lee2022convergence,lee2023convergence,chen2023sampling,conforti2023score,benton2024nearly,li2024towards,debortoli2022convergence},
traditional approaches based on Girsanov's theorem and path-space measures yield discretization error bounds of order $\mathcal{O}(h)$, which are not sharp. Furthermore, Girsanov techniques require identical diffusion coefficients between processes, making it hard to analyze deterministic flows (e.g., probability-flow ODEs \cite{chen2024probability,gao2023wasserstein,cheng2024convergence}) and stochastic samplers within a single path-space framework, thereby obscuring the fundamental trade-offs between them \cite{schaeffer2025stochasticity}.

In this work, we address these limitations by adopting a perspective from non-equilibrium statistical mechanics: the entropy production rate \cite{maes2003,seifert2005,qian2002thermodynamics,jiang2004measure,daCosta2023entropy}. If $\rho(\cdot,t)$ denotes the forward marginal and $\rho'(\cdot,T-t)$ the time-reversed marginal of a candidate backward process, the rate at which their relative entropy changes,
\[
    e_p(t)=-\frac{\di}{\di t}\,\DKL\big(\rho(\cdot,t)\,\|\,\rho'(\cdot,T-t)\big),
\]
is the entropy production rate of the pair, which vanishes identically if and only if the matching is perfect. The KL divergence between the true data distribution and the generated distribution is simply the time integral of this rate along the generative trajectory; related entropic accounts of diffusion generation appear in \cite{premkumar2023generative,premkumar2024neural}. We establish (Theorem~\ref{thm:mother-identity}) that this rate admits a closed-form identity pairing the score discrepancy of the two processes with the sum of their current velocities that
\[
    e_p(t)=-\int_{\mathbb R^n}\rho\,
    \left(\nabla\log\rho-\nabla\log\rho'\right)\cdot\big(v_c+v_c'\big)\,\di x .
\]
Because its derivation relies solely on the continuity equations of the marginal flows, this identity holds broadly for state-dependent and unequal diffusion matrices.

By pairing the forward process $(b,A)$ with a candidate backward process $(b',A')$, our framework uses the entropy production rate as a unified accumulator. Integrated along the reverse dynamics, it expresses the total generation error as the sum of an initialization mismatch, a score approximation error, and a discretization error. This error accounting relies on standard hypotheses: a finite Fisher information integral \cite{conforti2023score,cattiaux2022time} or, in its weakened early-stopping form, a finite second moment of the data \cite{chen2023sampling,benton2024nearly}.

The primary strengths of this framework are threefold.

\textbf{Consistency with existing bounds.} In appropriate special cases, our framework reduces to the standard three-part KL bound \cite{chen2023sampling,benton2024nearly}, the linear SDE stochasticity trade-off \cite{schaeffer2025stochasticity}, and the schedule-invariant log-signal-to-noise form of the variational bound \cite{kingma2021variational,gao2024diffusion}, serving as natural cross-checks.

\textbf{Sharper discretization estimates and broader generality.} By accumulating entropy production directly at the level of marginal distributions rather than in path space, the framework bypasses the path-to-marginal slack inherent in Girsanov-type change-of-measure arguments \cite{chen2023sampling,conforti2023score}. This yields an $\mathcal{O}(h^2)$ discretization error bound, in line with weak backward error analysis for SDEs \cite{zygalakis2011,debussche2012weak,High_Weak_Order_Methods,bea_diffusion_2023} and with the numerical irreversibility viewpoint of \cite{katsoulakis2013measuring}, and which is directly verifiable numerically. Moreover, the underlying identity accommodates state-dependent and unequal diffusions ($A\neq A'$) \cite{li2025estimates}, capturing the generality required to analyze discretized reverse flows.

\textbf{A unified treatment of diverse samplers.} Because every sampler in the field corresponds to a specific choice of $(b,A)$ and $(b',A')$, this framework offers a unified explanation for several empirical practices such as the existence of an optimal intermediate noise level \cite{karras2022elucidating,schaeffer2025stochasticity}, non-uniform step schedules \cite{karras2022elucidating,nichol2021improved}, prediction parameterizations \cite{salimans2022progressive,kingma2021variational,karras2022elucidating}, and early stopping \cite{song2021score,debortoli2022convergence}. Each of these phenomena naturally emerges from where entropy production is concentrated and how score error is weighted.
\subsection{Relation to existing work}
The pieces assembled here are individually established and we claim no priority
on them. The characterization of time irreversibility by the relative entropy
between the forward and time reversed path measures is classical in stochastic
thermodynamics \cite{maes2003,daCosta2023entropy,seifert2005}, and has been used
in generative and discrete diffusion models; in particular, the entropy
produced along the diffusion has been proposed as a measure of the
information a diffusion model stores, and score matching itself admits an
action principle formulation \cite{premkumar2023generative,premkumar2024neural,raquepas2024large,wu2025computing}.
The three part KL bound under a
finite relative Fisher information and an $L^2$ score error is the standard
analysis \cite{chen2023sampling,conforti2023score,benton2024nearly}; the
unification of diffusion, flow matching, and stochastic interpolants as a
tunable diffusion family is the stochastic-interpolant framework
\cite{albergo2023stochastic} and the diffusion / Gaussian-flow-matching
equivalence \cite{gao2024diffusion}; the schedule-admissibility relation and the
schedule-invariant log-signal-to-noise form of the variational bound are likewise
known \cite{gao2024diffusion,kingma2021variational}; the weak backward error
analysis of stochastic integrators is classical
\cite{zygalakis2011,debussche2012weak}; and the adjustable-diffusion family,
with the error-correction-versus-amplification trade-off of stochasticity, is
analyzed in KL by \cite{schaeffer2025stochasticity} under a linear forward SDE
with a Lipschitz score and no early stopping. For Gaussian data, exact
solutions of the reverse SDE and of the probability-flow ODE, together with
the exact Wasserstein error of arbitrary discretizations, are computed by
\cite{pierret2025diffusion}; their fully explicit account of the same error
sources: initialization, truncation, discretization, and score
approximation which is complementary to the KL-based bounds developed here, and
we use the same exactly solvable Gaussian setting in our numerical
experiments.

The individual ingredients we use are largely known and we attribute them as we go;
our aim is the consolidation and the perspective.

\section{Background}\label{sec:background}
    We begin our analysis with score-based diffusion generative models, one of the foundational theoretical frameworks for generative diffusion models. 
    The key idea of score-based diffusion models is to corrupt the original data via a forward diffusion process and then reconstruct the data by approximating the time reversal of this forward process.  This procedure can be formally described by a pair of SDEs \cite{song2021score}.
    
    \subsection{Forward diffusion process}
    Consider a measurable initial distribution (image or video data distribution)  $\mu_0 \in \mathcal{P}(\mathbb{R}^n)$, which is gradually corrupted by a diffusion process $\{X_t\}_{t \in [0,T]} \in \mathbb{R}^n$ satisfying the following SDE 
    \begin{equation}\label{forward diffusion process}
        \di X_t = b(X_t,t)\di t + \sigma(X_t,t)\di W_t,\quad X_0\sim \mu_0,
    \end{equation}
where $b(x,t):\mathbb{R}^n\times [0,T] \to \mathbb{R}^n$ is the drift term, $\sigma(x,t):\mathbb{R}^n\times[0,T] \to \mathbb{R}^{n\times n}$ is the diffusion matrix, and $W_t$ is the standard Wiener process (Brownian motion). Common choices for the forward process include the Ornstein-Uhlenbeck (OU) process and its kinetic counterpart.

We impose the following assumption:
\begin{assumption}\label{b-and-A}
The initial distribution $\mu_0$ admits a probability density function $\rho(x,0)$. The drift term and diffusion term $b, \sigma \in C^{k}(\mathbb{R}^n,[0,T])$ are smooth enough for some large integer $k$. The diffusion coefficient $A = \frac{\sigma\sigma^T}{2}$ satisfies the uniform ellipticity condition
\begin{equation*}
    \sum_{i,j =1}^na^{ij}v_iv_j \ge r\sum_{i=1}^n v_i^2,\quad \forall \bm v =(v_1,v_2,\dotsc,v_n) \in \mathbb{R}^n,\quad t\in [0,T],
\end{equation*} 
for some constant $r>0$.
\end{assumption}

Under these conditions, there exists a smooth probability density function $\rho(x,t)$ such that, for all $t \in [0,T]$, $\rho$ is the unique solution to the Fokker-Planck equation:
    \begin{equation}\label{Fokker plank equation}
         \partial_t \rho (x,t) = \mathcal{L}\rho(x,t):= -\nabla \cdot(b(x,t)\rho(x,t)) + \nabla\nabla:(A(x,t)\rho(x,t)).
     \end{equation}

\subsection{Backward diffusion process}

The goal of a generative model is to restore the data distribution from noise. 
Given the solution $\{X_t\}_{t \in [0,T]}$, we construct a backward stochastic process $\{X'_t\}_{t\in [0,T]}$ with initial distribution $\rho'_0 \approx \rho(\cdot ,T) $ and choosen parameters $b'$ and $\sigma'$
\begin{equation}\label{backward diffusion process}
    \di X'_t = b'(X'_t,t)\di t + \sigma'(X'_t,t)\di W_t,\quad X'_0 \sim \rho'_0 ,
\end{equation}
such that $X_0$ can be recovered from $X'_T$.

One natural approach is to construct $X'_t$ as the time reversal $X_t^-:= X_{T-t}$. However, the time reversal of a diffusion process does not generally retain the diffusion property (see Anderson, 1982, as cited in \cite{anderson1982reverse}). Prior work has established a necessary and sufficient condition for a diffusion process to remain a diffusion process under time reversal \cite{haussmann1986time, millet1989integration}. The result is stated as follows:

 \begin{assumption}\label{Fokker-Planck assumption}
Assume that the density of solution $\rho(x,t)$ satisfies the finite Fisher-information-type quantity i.e. there exists a constant $M_T$, such that
\begin{equation}\label{fisher information}
    \int_0^T\int_{\mathbb{R}^n}|\nabla\log\rho(x,t)|^2\rho(x,t)\di x\di t \le M_T< +\infty.
\end{equation}
 \end{assumption}
 
 \begin{proposition}\label{prop:time-reversal}
 Suppose the coefficients in \eqref{forward diffusion process} satisfy Assumption \ref{b-and-A}, and the solution satisfies Assumption \ref{Fokker-Planck assumption}.
 Then, $X_t^-$ is a diffusion process that satisfies the following SDE:
 \begin{equation*}
     \di X^-_t = b^-(X_t^-,t)\di t + \sigma^-(X_t^-,t)\di W_t,\quad X^-_0  = X_T \sim \rho(x,T) ,
 \end{equation*}
 where
 \begin{equation}\label{reverse_drift_relation}
     \begin{aligned}
          &b^{-}(x,t) = \begin{cases} - b(x,T-t), & \rho(x,T-t) = 0,\\[2pt]
          - b(x, T-t) +\frac{2}{\rho}\,\nabla\cdot\big(A\rho\big)(x,T-t), & \rho(x,T-t) >0,\end{cases}\\
        &\sigma^{-}(t) = \sigma(T-t),
     \end{aligned}
 \end{equation}
 and $X_t^-$ is called the time-reversed process.
\end{proposition}

\subsection{Weakened assumptions via early stopping}\label{sec:early-stopping}
However, the finite-Fisher-information condition \eqref{fisher information} of
Assumption \ref{Fokker-Planck assumption}, taken globally on $[0,T]$, is not
satisfied by the data distributions
that motivate diffusion models: quantized images carry no density, and
manifold-supported or boundary-singular data have $I(\mu_0)=+\infty$.

In the following we show that the global condition can be replaced by a
finite second moment of the data together with an early-stopping
level $\delta\in(0,T)$. 

\begin{proposition}
\label{prop:malliavin-cap}
Let the coefficients $b(x,t)$ and $\sigma(x,t)$ of
\eqref{forward diffusion process} be as in Assumption
\ref{b-and-A}
. Then, for every
$\mu_0\in\mathcal{P}(\mathbb{R}^n)$ (no density and no moments assumed), the
marginal $\rho_t$ has a smooth, strictly positive density, and there is a
constant $C$ depending only on the coefficient bounds such that
\begin{equation}\label{eq:malliavin-cap}
    I(\rho_t)\ \le\ \frac{C\,n}{r\,t}\quad (0<t\le1),
    \qquad
    I(\rho_t)\ \le\ \frac{C\,n}{r}\quad (t\ge1).
\end{equation}
Consequently,
\begin{equation}\label{eq:malliavin-M}
    M_{\delta,T}=\int_\delta^T I(\rho_t)\,\di t
    \ \le\ \frac{C\,n}{r}\Big(\log\frac1\delta+T\Big),
    \qquad 0<\delta\le\min\{T,1\}.
\end{equation}
\end{proposition}
\begin{proof}
Let $Y_t = \frac{\partial X_t}{\partial X_0}$ be the first variation process of the SDE. The Malliavin derivative of $X_t$ with respect to  the Brownian motion ($s \le t$) is given by
\begin{equation}\label{Malliavin_Derivative}
    D_s X_t = Y_t Y_s^{-1}\sigma(X_s,s) \mathbf{1}_{\{s \le t\}}.
\end{equation}
The Malliavin covariance matrix $\gamma_t$ is defined as
\begin{equation}\label{Malliavin covariance matrix}
    \gamma_t = \int_0^t (D_s X_t)(D_s X_t)^T \di s = Y_t \left( \int_0^t Y_s^{-1}\sigma(X_s,s)\sigma(X_s,s)^T(Y_s^{-1})^T \di s \right) Y_t^T.
\end{equation}

Under the uniform ellipticity of $\sigma\sigma^T$ and the smoothness of the coefficients (Assumption \ref{b-and-A}), $\gamma_t$ is invertible almost surely. Moreover, its inverse has bounded moments of all orders (see, e.g., \cite[Chapter~2]{nualart2006malliavin} or \cite{kusuoka1985applications})
\begin{equation}\label{gamma_t estimation}
    \mathbb{E}\big[\|\gamma_t^{-1}\|^p\big] \le \frac{C_p}{t^p}, \quad \forall p \ge 1.
\end{equation}

For any test function $\phi \in C_c^\infty(\mathbb{R}^n)$, the chain rule yields $D_s(\phi(X_t)) = \nabla \phi(X_t) D_s X_t$. Multiplying both sides by $(D_s X_t)^T \gamma_t^{-1}$ and integrating over $s \in [0, t]$, we obtain
\begin{equation*}
    \int_0^t D_s(\phi(X_t)) (D_s X_t)^T \gamma_t^{-1} \di s = \nabla \phi(X_t) \gamma_t \gamma_t^{-1} = \nabla \phi(X_t).
\end{equation*}

Taking the expectation and applying the integration-by-parts formula on the Wiener space (the duality relationship between the Malliavin derivative $D$ and the Skorokhod divergence operator $\delta$), we have
\begin{equation}
    \mathbb{E}[\nabla \phi(X_t)] = \mathbb{E}\left[ \int_0^t D_s(\phi(X_t)) \cdot \left((D_s X_t)^T \gamma_t^{-1}\right) \di s \right] = \mathbb{E}[\phi(X_t) H_t],
\end{equation}
where $H_t := \delta\big( (D_\cdot X_t)^T \gamma_t^{-1} \big)$ is the Malliavin weight. 
(Note that the Skorokhod integral $\delta$ is required here instead of the standard It\^o integral, because $\gamma_t^{-1}$ depends on the future paths up to $t$ and is not adapted to the natural filtration $\mathcal{F}_s$).

On the other hand, applying classical integration by parts in $\mathbb{R}^n$ yields
\begin{equation}
\begin{aligned}
    \mathbb{E}[\nabla \phi(X_t)] &= \int_{\mathbb{R}^n} \nabla \phi(x) \rho_t(x) \di x = -\int_{\mathbb{R}^n} \phi(x) \nabla \log \rho_t(x) \rho_t(x) \di x \\ &= -\mathbb{E}\big[\phi(X_t) \nabla \log \rho_t(X_t)\big].
\end{aligned}
\end{equation}

Since this holds for any smooth test function $\phi$, we can identify the score function as the conditional expectation of the Malliavin weight
\begin{equation}\label{Malliavin weight}
    \nabla \log \rho_t(X_t) = -\mathbb{E}[H_t \mid X_t].
\end{equation}

Finally, by applying Jensen's inequality for conditional expectations, the Fisher information of $\rho_t$ is bounded by the variance of $H_t$:
\begin{equation}
    I(\rho_t) = \mathbb{E}\big[|\nabla \log \rho_t(X_t)|^2\big] = \mathbb{E}\big[|\mathbb{E}[H_t \mid X_t]|^2\big] \le \mathbb{E}\big[|H_t|^2\big].
\end{equation}
Standard Meyer inequalities in Malliavin calculus, combining the bounds on the Malliavin derivatives and the inverse covariance matrix \eqref{gamma_t estimation}, yield $\mathbb{E}[|H_t|^2] \le \frac{Cn}{rt}$ \cite{kusuoka1985applications}. Consequently,
\begin{equation}
    I(\rho_t) \le \frac{Cn}{rt},
\end{equation}
which finishes the proof.
\end{proof}

Now we can replace the global Assumption \ref{Fokker-Planck assumption} by the
following weaker hypothesis.
 
\begin{assumption}\label{assump:second-moment}
The data distribution satisfies $m_2=\mathbb{E}|X_0|^2<+\infty$. 
\end{assumption}

Under early stopping the sampler no longer targets $\mu_0$ itself but the
mollified data distribution
\begin{equation}\label{eq:mollified-target}
    \rho_\delta=\operatorname{law}(X_\delta).
\end{equation}
The
gap between $\rho_\delta$ and the true target $\rho_0$ is quantified in the
Wasserstein-$2$ distance by Lemma \ref{lem:mollification-gap} in
Section \ref{sec:error}; the Kullback--Leibler divergence is not the right
metric for this gap, since $\DKL(\rho_0\,\|\,\cdot)=+\infty$ whenever
$\rho_0$ is singular. The early-stopped counterpart of the total error bound
is stated in Corollary \ref{cor:early-stopped}.

\begin{example}
\label{ex:ou-dirac}
Take the OU process $\di X_t=-\tfrac12 X_t\,\di t+\di W_t$ with deterministic
initial condition $X_0\equiv m$ (the Dirac case $\mu_0=\delta_m$). Then
$\alpha_t=e^{-t/2}$, $\beta_t^2=1-e^{-t}$, and
\[
    \rho_t=\mathcal N\!\big(e^{-t/2}m,\,(1-e^{-t})I\big),\qquad
    I(\rho_t)=\frac{n}{1-e^{-t}} .
\]
The global Fisher integral diverges logarithmically at the data end,
\[
    \int_0^T I(\rho_t)\,\di t
    =n\int_0^T\frac{\di t}{1-e^{-t}}=+\infty
    \qquad\text{(since }1-e^{-t}\sim t\text{ as }t\to0\text{)},
\]
so $\mu_0$ violates the global condition \eqref{fisher information}. Yet the truncated integral is finite with
an explicit logarithmic constant,
\[
    M_{\delta,T}=n\int_\delta^T\frac{\di t}{1-e^{-t}}
    =n\,\big[\log(e^t-1)\big]_\delta^T
    =n\log\frac{e^T-1}{e^\delta-1}
    \ \le\ n\Big(\log(e^T-1)+\log\tfrac1\delta\Big),
\]
in agreement with Proposition \ref{prop:malliavin-cap}. The mollification gap is computed
directly: with $m_2=|m|^2$,
\[
    W_2^2(\delta_m,\rho_\delta)\le(1-e^{-\delta/2})^2|m|^2+(1-e^{-\delta})n
    \ \sim\ \tfrac{\delta^2}{4}|m|^2+\delta n
    \ =\ \mathcal O(\delta n)\qquad(\delta\to0),
\]
an instance of Lemma \ref{lem:mollification-gap} below. The example also
quantifies the cost of the singularity at the data end: the reverse drift has
linear coefficient $\tfrac12-\tfrac1{1-e^{-t}}\sim-\tfrac1t$ as $t\to0$, so
its time derivative and self-products whose ingredients of the
backward-error-analysis correction of Section \ref{sec:error} scale like
$t^{-2}$ and $t^{-3}$. The Fisher contribution to the early-stopped constants
is therefore logarithmic in $\frac{1}{\delta}$, while the discretization constant is
polynomial in $\frac1\delta$;
\end{example}

\subsection{Score based diffusion model}
 The distributional derivative $\nabla \log  \rho_t$, known as the Stein score function, is in general difficult to estimate.

 The diffusion model is thus revised to learn this score via neural network function approximation. Formally, the score function is approximated by minimizing the loss function $\mathcal{L}(\theta)$:
\begin{equation}\label{eq:optimzation_loss_function}
\begin{aligned}
    \min_\theta \mathcal{L}(\theta) = &\min_\theta \mathbb{E}_{t \in \mathcal{U}[0,T]}\mathbb{E}_{X_t}\left[\|s_\theta(X_t,t) - \nabla \log \rho(X_t,t)\|^2\right] \\= &\min _{\theta}\|s_\theta(x,t) - \nabla \log \rho(x,t)\|_{L^2_\rho(\mathbb{R}^n;[0,T])},
    \end{aligned}
\end{equation}
where $ \{s_\theta(x,t)\}:\mathbb{R}^n\times[0,T] \to \mathbb{R}^n$ denotes a family of functions parametrized by $\theta$ in a neural network.

Accordingly, the backward diffusion process can be constructed by
\begin{equation}\label{score matching diffusion process}
    \di X^{\theta,-}_t = \left(-b(X^{\theta, -}_t,T-t) + 2A^Ts_\theta(X^{\theta,-}_t,T-t)\right)\di t + \sigma(X^{\theta,-}_t, t)\di W_t.
\end{equation}

For training convenience, we typically require that the solution $\{X_t\}_{t\in [0,T]}$ admits a closed-form expression $X_t= \Psi(X_0,t,\epsilon)$, where $\epsilon$ is a random variable. In typical settings, we impose that
\begin{equation}\label{eq:affine-marginal}
    X_t = \alpha_t X_0 + \beta_t \epsilon,\quad \epsilon \sim \mathcal{N}(0,I),
\end{equation}
which induces various noise schedules such as linear, cosine, exponential, and Karras schedules. 

Thus, the loss function \eqref{eq:optimzation_loss_function} can be rewritten as
\begin{equation}\label{eq:epsilon_prediction}
    \min_{\theta}\mathcal{L}(\theta) = \min_\theta \mathbb{E}_{t\sim \mathcal{U}(0,T)}\mathbb{E}_{\epsilon \sim \mathcal{N}(0,I)}\|\hat{s}_\theta\left(\Psi(X_0,t,\epsilon),t\right) - \epsilon\|_2^2,
\end{equation}
which corresponds to the $\epsilon$-prediction formulation.

Meanwhile, for general SDE, we can learn the score function by Malliavin weight
\begin{equation}
        \nabla \log \rho_t(X_t) = -\mathbb{E}[H_t \mid X_t].
\end{equation}
\subsection{Flow matching and stochastic interpolants}
\label{sec:interpolant-background}

Score-based diffusion is not the only way to connect the data distribution
to a tractable base distribution. A stochastic interpolant
\cite{albergo2023stochastic} prescribes the bridge directly, as a path of
random variables
\begin{equation}\label{eq:interpolant-def}
    X_t = \alpha_t X_0 + \beta_t X_1 + \gamma_t\,\xi,\qquad
    \xi\sim\mathcal N(0,I),
\end{equation}
where $X_0\sim\mu_0$ is a data sample, $X_1$ is a base sample (typically
Gaussian), $\xi$ is an independent smoothing noise, and the schedule
$(\alpha_t,\beta_t,\gamma_t)$ satisfies the boundary conditions
$\alpha_0 =1$, $\beta_0 =0$ and $\alpha_T= 0$, $\beta_T =1$, such that the path
interpolates between data and base. Flow matching
\cite{lipman2023flow,liu2023rectified} generates samples by learning the
velocity of this interpolation, regressing $\dot X_t$ on $X_t$, and
integrating the resulting ODE; the precise object this regression targets,
and its identification with the current velocity of the marginal flow, are
given in Section \ref{sec:unified-comparison}.

When the base is Gaussian and
$\gamma\equiv0$, \eqref{eq:interpolant-def} reduces to the family
$X_t=\alpha_t X_0+\beta_t\xi$, which can be seen as the closed-form expression of affine SDE. We therefore work with the two-coefficient schedule
$(\alpha_t,\beta_t)$ in what follows.

However, not every schedule $(\alpha_t,\beta_t)$ is the marginal law of a
forward SDE. Such a schedule does not admit a time reversal within the
diffusion family, and the missing dissipation can amplify the errors of the
learned fields rather than damp them. We therefore restrict attention
throughout to admissible schedules, characterized as follows.

\begin{proposition}\label{prop:admissible-schedule}
A Gaussian interpolant schedule $(\alpha_t,\beta_t)$ with
$\alpha_t,\beta_t>0$ and marginals $x_t=\alpha_t x_0+\beta_t\,\xi$, where
$x_0\sim\mu_0$ and $\xi\sim\mathcal N(0,I)$ are independent, is the marginal
law of an affine forward SDE
\begin{equation*}
    \di x=f(t)x\,\di t+g(t)\,\di W_t
\end{equation*}
if and only if
\begin{equation}\label{eq:legality}
    g^{2}(t)=2\beta_t^{2}\Big(\frac{\dot\beta_t}{\beta_t}-\frac{\dot\alpha_t}{\alpha_t}\Big)\ \ge\ 0
    \qquad\text{for all }t\in [0,T],
\end{equation}
equivalently when the signal-to-noise ratio $\text{SNR}(t):=\alpha_t/\beta_t$ is non-increasing
in the forward direction, in which case $f=\dot\alpha_t/\alpha_t$ and
$A(t)=\tfrac12 g^{2}(t)$.
\end{proposition}

\subsection{Error measurement}
Two standard metrics are used to measure the error: the Wasserstein distance and the KL-divergence. In this section, we give the definitions as follows.
\begin{definition}
Let $\Pi(P, Q)$ be the set of all joint probability distributions $\gamma(x, y)$ on $\mathcal{X} \times \mathcal{X}$ with marginal distributions $P$ and $Q$, respectively. That is, for any $\gamma \in \Pi(P, Q)$: $\int \gamma(x, y) dy = p(x)$, $\int \gamma(x, y) dx = q(y)$. The $p$-Wasserstein distance is defined as the minimum cost over all possible transport plans:
$$
W_p(P, Q) = \left( \inf_{\gamma \in \Pi(P, Q)} \int_{\mathcal{X} \times \mathcal{X}} d(x, y)^p \, d\gamma(x, y) \right)^{1/p}.
$$

For $p=1$, the definition simplifies to the classic Earth Mover's Distance:
$$
W_1(P, Q) = \inf_{\gamma \in \Pi(P, Q)} \mathbb{E}_{(x, y) \sim \gamma} [d(x, y)].
$$
\end{definition}

The KL divergence quantifies how a probability distribution $P$ deviates from a reference probability distribution $Q$.

\begin{definition}
Let $P$ and $Q$ be probability distributions on a sample space $\mathcal{X}$ with probability mass functions $P(x)$ and $Q(x)$. The KL divergence is defined as 
$$
\DKL(P || Q) = \mathbb{E}_{x \sim P} \left[ \log \left( \frac{P(x)}{Q(x)} \right) \right] = \sum_{x \in \mathcal{X}} P(x) \log \left( \frac{P(x)}{Q(x)} \right).
$$
For continuous distributions with probability density functions (PDFs) $p(x)$ and $q(x)$:
$$
\DKL(P || Q) = \mathbb{E}_{x \sim p} \left[ \log \left( \frac{p(x)}{q(x)} \right) \right] =\int p(x) \log \left( \frac{p(x)}{q(x)} \right) dx.
$$
\end{definition}

In this work, we employ tools from non-equilibrium statistical dynamics to quantify error.

For a finite-time pair of forward and backward processes, we use the entropy production rate $e_p$ to characterize the decrease of relative entropy along the comparison:
\begin{equation}\label{eq:entropy_production_rate_def}
    e_p(t) = -\frac{\di}{\di t}\DKL(\rho_t||\rho'_{T-t}),\quad t\in [0,T],
\end{equation}
where the relative entropy is given by $\DKL(\rho_t||\rho'_{T-t})(t) = \int_{\mathbb{R}^n}\rho(x,t)\log \frac{\rho(x,t)}{\rho'(x,T- t)}\di x$. With this convention, positive $e_p$ contributes positively to the terminal KL error.

The discrepancy between the generated distribution $X'_T$ and the target distribution $X_0$ can then be measured via the KL divergence, which satisfies the following identity
\begin{equation}\label{eq:KL-divergence_equality}
    \int_0^Te_p(t) \di t =\DKL(X_0||X'_T) - \DKL(X_T||X'_0) .
\end{equation}
For a reversible diffusion process, it follows that $e_p = 0$ when $\rho'_{T-t} \equiv \rho^-$ (i.e., $X'_t$ = $X^-_t$). In this case \eqref{eq:KL-divergence_equality} only implies $\DKL(X_0||X'_T) =\DKL(X_T||X'_0)$. If, in addition, the backward process is initialized from the exact terminal law $X'_0\sim X_T$, then $\DKL(X_0||X'_T) = 0$.

\section{Entropy production rate of a pair of diffusion processes}
\label{sec:Entropy}
 
In this section we derive the central identity of the paper: a closed-form
expression for the entropy production rate $e_p$ of a forward process against
a comparison backward process. We obtain it in a form that holds for
state and time-dependent diffusion matrices and that does
not require the two processes to share the same diffusion matrix. 
 
\subsection{From Fokker--Planck to continuity equations}
\label{sec:continuity}
 
Let $\{X_t\}_{t\in[0,T]}$ be the forward process \eqref{forward diffusion process}
with drift $b(x,t)$ and diffusion matrix $A(x,t)=\tfrac12\sigma\sigma^T(x,t)$. Its density solves the Fokker--Planck
equation
\begin{equation}\label{eq:fp-general}
    \partial_t\rho_t=-\nabla\cdot(b\rho_t)+\nabla\nabla:(A\rho_t),
\end{equation}
where $\nabla\nabla:(A\rho_t)=\sum_{i,j}\partial_i\partial_j(A_{ij}\rho_t)$.
Equation \eqref{eq:fp-general} is a continuity equation $\partial_t\rho_t=-\nabla\cdot(v_c\rho_t)$
for the forward current velocity
\begin{equation}\label{eq:fwd-velocity}
    v_c(x,t) := b-\frac{1}{\rho_t}\nabla\cdot(A\rho_t)
    =\frac{(b -b^-)}{2}
\end{equation}
which is the probability-flux velocity $J=v_c\rho$ of the forward process.
 
Let $\{X'_t\}_{t\in[0,T]}$ be the comparison backward process
\eqref{backward diffusion process} with its own drift $b'(x,t)$ and diffusion
matrix $A'(x,t)=\tfrac12\sigma'\sigma'^{T}(x,t)$.
Similarly we have the comparison backward current velocity 
\begin{equation}\label{eq:bwd-velocity}
    v'_c(x,t) := b'-\frac{1}{\rho'_{T-t}}\nabla\cdot(A'\rho'_{T-t})
    =\frac{(b' -b'^-)}{2},
\end{equation}
where, here and throughout this section, primed quantities are evaluated at
reverse time $T-t$. In the second equalities of
\eqref{eq:fwd-velocity}--\eqref{eq:bwd-velocity}, $b^-$ and $b'^-$ denote the
reverse drifts of Proposition \ref{prop:time-reversal}, evaluated at the
matching reverse time, so that the current velocity appears as the
time-antisymmetric part of the drift pair.

The two velocities $v_c$ and $v_c'$ are the only objects that enter the entropy
production rate. We never invert $A$, and the formulas \eqref{eq:fwd-velocity}--\eqref{eq:bwd-velocity}
remain valid when $A$ depends on $x$ and when $A'\neq A$.
 
\begin{assumption}\label{assump:ep-regularity}
On $[0,T]$ the densities $\rho,\rho'$ are positive and $C^{2,1}$ in $(x,t)$;
$\rho'(\cdot,T-t)>0$ wherever $\rho(\cdot,t)>0$; the velocities
$v_c,v_c'$ in \eqref{eq:fwd-velocity}--\eqref{eq:bwd-velocity} are locally
integrable against $\rho$; and the relative-entropy integrand and its time
derivative are dominated by an integrable function uniformly on compact
$t$-intervals, so that differentiation under the integral sign is justified.
Moreover the boundary fluxes
$\log\tfrac{\rho_t}{\rho'_{T-t}}(v_c\rho_t)$ and
$\tfrac{\rho_t}{\rho'_{T-t}}\,\big(v_c'\rho'_{T-t}\big)$ vanish as
$|x|\to\infty$, so the integrations by parts below carry no boundary terms.
\end{assumption}
 
\begin{remark}
Assumption \ref{assump:ep-regularity} is automatically satisfied in the
smooth uniformly elliptic setting of Assumption \ref{b-and-A} together with
the finite-Fisher-information condition of Assumption
\ref{Fokker-Planck assumption}, or, after early stopping, on every interval
$[\delta,T]$ under the weaker hypotheses of Section \ref{sec:early-stopping}.
The point of stating it abstractly is that the identity below does not need
ellipticity or $A'=A$; it needs only two well-posed continuity equations.
\end{remark}

\begin{theorem}[Entropy production rate]\label{thm:mother-identity}
Under Assumption \ref{assump:ep-regularity}, the entropy production rate of
the forward process \eqref{forward diffusion process} against the comparison
process \eqref{backward diffusion process} is
\begin{equation}\label{eq:mother-identity-current}
    e_p(t)= -\int_{\mathbb R^n}
    \rho_t\nabla\log\frac{\rho_t}{\rho'_{T-t}}\cdot
    \Big(v_c+v'_c\Big)\di x .
\end{equation}
\end{theorem}
 
\begin{proof}
Write $D(t)=\DKL\big(\rho_t\,\|\,\rho'_{T-t}\big)=\int\rho_t\log\tfrac{\rho_t}{\rho'_{T-t}}\,\di x$.
Differentiating and using $\int\partial_t\rho\,\di x=0$,
\begin{equation}\label{eq:dD-split}
    \frac{\di}{\di t}D(t)
    =\int \partial_t\rho_t\,\log\frac{\rho_t}{\rho'_{T-t}}\,\di x
    -\int \frac{\rho_t}{\rho'_{T-t}}\,\partial_t\rho'_{T-t}\,\di x .
\end{equation}
For the first term, substitute $\partial_t\rho=-\nabla\cdot(v_c\rho)$ and
integrate by parts (no boundary term by Assumption
\ref{assump:ep-regularity}):
\[
    \int \log\frac{\rho_t}{\rho'_{T-t}}\,\left(-\nabla\cdot(v_c\rho_t)\right)\di x
    =\int \nabla\log\frac{\rho_t}{\rho'_{T-t}}\cdot v_c\rho_t\di x .
\]
For the second term, substitute $\partial_t\rho'_{T-t}=\nabla\cdot(v_c'\rho'_{T-t})$
and integrate by parts:
\begin{align*}
    \int\frac{\rho_t}{\rho'_{T-t}}\big(\nabla\cdot(v_c'\rho'_{T-t})\big)\di x
    &=-\int\nabla\!\Big(\frac{\rho_t}{\rho'_{T-t}}\Big)\cdot v_c'\rho'_{T-t}\di x\\
    &=-\int \frac{\rho_t}{\rho'_{T-t}}\,\nabla\log\frac{\rho_t}{\rho'_{T-t}}\cdot v_c'\rho'_{T-t}\di x\\
    &=-\int \rho_t\,\nabla\log\frac{\rho_t}{\rho'_{T-t}}\cdot v_c'\di x,
\end{align*}
where we used $\nabla(\rho_t/\rho'_{T-t})=(\rho_t/\rho'_{T-t})\nabla\log(\rho_t/\rho'_{T-t})$.
Adding the two contributions,
\[
    \frac{\di}{\di t}D(t)=\int \rho_t\,\nabla\log\frac{\rho_t}{\rho'_{T-t}}\cdot(v_c+v_c')\,\di x,
\]
and \eqref{eq:mother-identity-current} follows from
$e_p=-\tfrac{\di}{\di t}D$.
\end{proof}

\subsection{Singular case}
\label{sec:singular}
 
The identity above is stated at the density level and presumes well-defined,
positive marginals (Assumption \ref{assump:ep-regularity}). A distinct
obstruction appears for degenerate diffusions, where the relevant quantities
need not be finite even though the formal expression makes sense.
 
\begin{remark}\label{rem:singular-path-kl}
Compare the forward process and a time-reversed comparison process started at
time $t$,
\[
    \di X_s=b(X_s,s)\,\di s+\sigma(s)\,\di W_s,\qquad
    \di Y_s=b'^{-}(Y_s,s)\,\di s+\sigma(s)\,\di\tilde W_s,\quad s\in[t,T].
\]
Let $\mathbb P_X,\mathbb Q_Y$ be the induced path measures. By Girsanov's
theorem, $\mathbb Q_Y\ll\mathbb P_X$ requires the drift difference
$\delta b=b-b'^{-}$ to lie in $\mathrm{Range}\,\sigma$ with the usual
integrability (e.g.\ Novikov). If $\delta b$ has a component outside
$\mathrm{Range}\,\sigma$ on a set of positive path measure, the two path
measures are mutually singular and the path-space KL divergence is infinite.
This is a path-space, not a fixed-time, statement: it does not by itself
force $\DKL\big(\rho(\cdot,t)\|\rho'(\cdot,T-t)\big)$ or the marginal rate $e_p(t)$
to be infinite, which would require a separate density-level argument through
\eqref{eq:fp-general} or a Fisher-information quantity involving the
pseudoinverse of $A$. In the uniformly elliptic setting of Assumption
\ref{b-and-A}, $\mathrm{Range}\,\sigma(t)=\mathbb R^n$ and this obstruction is
absent. We therefore treat the singular and degenerate cases as limitations
and extensions of the present regular theory (Section \ref{sec:conclusion}).
\end{remark}

\section{Error analysis of numerical scheme}\label{sec:error}

The practical implementation of score-based diffusion models requires the discretization of continuous-time SDEs, which introduces an additional discretization error.
 
In our work, the total error is decomposed into three components: the initialization error, the score approximation error, and the time discretization error. The common identity is the endpoint relation
\begin{equation}\label{eq:error-bookkeeping-template}
    \DKL(X_0||X'_T)
    =\DKL(X_T||X'_0)+\int_0^T e_p(t)\di t,
\end{equation}
with the sign convention of \eqref{eq:entropy_production_rate_def}. The term $\DKL(X_T||X'_0)$ measures the mismatch between the true terminal forward law and the initialization law used by the reverse sampler, while the integrated entropy-production term measures the accumulated mismatch between the ideal and practical velocity fields.

Analysis in \cite{conforti2023score} shows that for a score-based generative model based on the OU process, the KL divergence between the data distribution and the generated distribution can be bounded. This bound depends on the score approximation error, the Fisher information of the data distribution, and the step size. We point out that a smaller step size reduces the discretization error but increases the computational cost.

Our framework, based on the entropy production rate, can be extended to analyze error propagation in different numerical schemes. By analyzing how discretization affects the osmotic and current velocities, we can quantify the contribution of the numerical scheme to the entropy production rate at each step \cite{katsoulakis2013measuring}. This provides a more detailed understanding of how error accumulates over time and can guide the selection of the numerical scheme and step size to minimize the final KL divergence.

In this section, we present a regular-case error analysis of score-based generative diffusion models. We first describe the discretization contribution through backward error analysis, then the score-matching contribution, and finally the initialization contribution before combining them in Theorem \ref{thm:total-error}.
 
\subsection{Discretization error of the Euler--Maruyama scheme}
In practice, the backward process is discretized by a one-step numerical integrator with mesh size $h = T/N$, given by  
 \begin{equation*}
     X_{k+1} = \Phi(X_k,h,\xi_k),\quad k = 0,1,\dotsc,N-1,
 \end{equation*}
 to discretize the backward process \eqref{score matching diffusion process}. For instance, we use the Euler-Maruyama scheme. The question arises whether a numerical scheme such as the Euler-Maruyama scheme, which approximates the reversed process $X^-_t$, converges in terms of the KL divergence. In this section, we use the entropy production rate and backward error analysis to establish the convergence of the numerical scheme. An alternative approach to prove convergence exists, and the proof for the Euler-Maruyama scheme can be found in \cite{li2025estimates}.
 
Backward error analysis \cite{debussche2012weak} is a technique initially developed to explain the behavior of deterministic algorithms, which has recently regained interest 
for applications in studying stochastic algorithms, including the analysis of integrators for SDEs \cite{High_Weak_Order_Methods,katsoulakis2013measuring,wang2018computing,zhang2024structure,sprekeler2025numerical}. For a one-step numerical integrator $X_{k+1} = \Phi(X_k,h,\xi_k)$ approximating a target SDE 
\begin{equation*}
    \di X_t = f(X_t,t)\di t +g(X_t,t)\di W_t,
\end{equation*}
the idea of backward error analysis is to find a modified SDE
\begin{equation*}
     \di X^h_t = f_h(X_t^h,t) \di t + g_h(X_t^h,t)\di W_t,\quad X_0^h= X_0,
\end{equation*}
that more accurately describes the behavior of the numerical schemes.

\begin{theorem}\label{thm:modified-equation}Suppose $f(x,t) \in C^k_b(\mathbb{R}^n,[0,T])$, $g(t)\in C^k_b ([0,T])$, for the Euler-Maruyama scheme, we have the following first order modified equations $f_h = f + f_{1}h$ and $g_h = g + g_1 h$ with
\begin{equation}\label{modified equation formula}
 \begin{aligned}
    f_{1} &=\frac{1}{2}\partial_t f + \frac{1}{2}f\cdot \nabla f + \frac{1}{4}gg^T:\nabla\nabla f,\\
    g_1 &= \frac{1}{2}\partial_t g + \frac{1}{2}g\cdot \nabla f +\frac{1}{2}f\cdot \nabla g + \frac{1}{4}gg^T:\nabla\nabla gg^T,
    \end{aligned}
\end{equation}
such that for any $T>0$ and test function $\phi \in C_P^\infty(\mathbb{R}^n)$,
\begin{equation}\label{backward error analysis}
    \|\mathbb{E}\phi(X_k)  - \mathbb{E}\phi(X^h_{kh})\| = \mathcal{O}(h^2),\quad 0\le kh\le T.
\end{equation}
\end{theorem}
\begin{proof}
Considering the augmented state approach, we let 
    $Y_t = \left(X_t ,t\right)^T\in \mathbb{R}^{n +1}$, $F(Y_t) = \left(f(X_t,t), 1\right)^T$ and $G(Y_t) = \left(
   g(t),0\right)^T$. Then the result can be obtained by the proposition of \cite{High_Weak_Order_Methods}.
\end{proof}

\begin{lemma}\label{lem:modified-coefficients-stability}
Under the modified-equation expansion in Theorem \ref{thm:modified-equation}, suppose the coefficients of the target SDE satisfy $b,A\in C_b^k(\mathbb{R}^n\times[0,T])$ for sufficiently large $k$, and suppose $A$ is uniformly elliptic
\begin{equation*}
    \xi^T A(x,t)\xi \ge r|\xi|^2,\quad r>0.
\end{equation*}
Then the modified coefficients can be written as
\begin{equation*}
    b_h=b+h b_1+\mathcal{O}(h^2),\qquad
    A_h=A+h A_1+\mathcal{O}(h^2),
\end{equation*}
where $b_1$ and $A_1$ are determined by $b$, $A$, and finitely many of their derivatives. Consequently, for a possibly smaller integer $k'$,
\begin{equation}\label{eq:modified-coefficients-stability}
    \|b_h-b\|_{C^{k'}}+\|A_h-A\|_{C^{k'}}\le C h.
\end{equation}
In particular, defining the Fokker--Planck coefficient discrepancy
\begin{equation}\label{eq:fp-coefficient-discrepancy}
    \alpha_h:=\|b_h-b\|_{L^\infty(\mathbb{R}^n\times[0,T])}
    +\|A_h-A\|_{L^\infty(\mathbb{R}^n\times[0,T])}
    +\|\nabla\cdot(A_h-A)\|_{L^\infty(\mathbb{R}^n\times[0,T])},
\end{equation}
we have $\alpha_h\le Ch$ whenever $k'\ge 1$.
In particular, for $h$ small enough, $A_h$ remains uniformly elliptic
\begin{equation*}
    \xi^T A_h(x,t)\xi \ge \frac{r}{2}|\xi|^2.
\end{equation*}
\end{lemma}

\begin{proof}
The modified-equation construction expresses the first correction terms as finite combinations of $b$, $A$ and their derivatives; see \cite{debussche2012weak,High_Weak_Order_Methods}. Hence the $C^{k'}$ estimate \eqref{eq:modified-coefficients-stability} follows from the assumed $C_b^k$ bounds. If the modified equation is first written for the diffusion matrix $\sigma_h=\sigma+h\sigma_1+\mathcal{O}(h^2)$, then
\begin{equation*}
    A_h=\frac{1}{2}\sigma_h\sigma_h^T
    =A+\frac{h}{2}(\sigma_1\sigma^T+\sigma\sigma_1^T)+\mathcal{O}(h^2),
\end{equation*}
so $\|A_h-A\|_{L^\infty}\le Ch$. Therefore,
\begin{equation*}
    \xi^T A_h\xi
    \ge \xi^T A\xi-\|A_h-A\|_{L^\infty}|\xi|^2
    \ge (r-Ch)|\xi|^2.
\end{equation*}
Taking $h<r/(2C)$ gives the claimed ellipticity.
\end{proof}

\begin{lemma}\label{lem:weighted-score-stability}
Let $\rho$ and $\rho^h$ be positive smooth solutions of the forward Fokker--Planck equations
\begin{align*}
    \partial_t\rho &= -\nabla\cdot(b\rho)+\nabla\nabla:(A\rho),\\
    \partial_t\rho^h &= -\nabla\cdot(b_h\rho^h)+\nabla\nabla:(A_h\rho^h),
\end{align*}
with the same initial density $\rho(\cdot,0)=\rho^h(\cdot,0)$. Let
\begin{equation*}
    u_h(x,t):=\log \frac{\rho(x,t)}{\rho^h(x,t)} .
\end{equation*}
Assume that $A$ and $A_h$ are uniformly elliptic with common ellipticity constant $r>0$, and let $\alpha_h$ be the coefficient discrepancy defined in \eqref{eq:fp-coefficient-discrepancy}. Assume further the finite Fisher-information-type condition
\eqref{fisher information} of Assumption \ref{Fokker-Planck assumption},
\begin{equation*}
    \int_0^T\int_{\mathbb{R}^n}|\nabla\log\rho(x,t)|^2\rho(x,t)\di x\di t \le M_T< +\infty.
\end{equation*}
Then there exists $C_T>0$, depending on the ellipticity and the uniform coefficient bounds, such that
\begin{equation}\label{eq:weighted-score-stability}
    \left\|\nabla\log\left(\frac{\rho}{\rho^h}\right)\right\|_{L^2_\rho(\mathbb{R}^n;[0,T])}
    \le C_T\alpha_h ,
\end{equation}
and the relative entropy satisfies
\begin{equation}\label{eq:relative-entropy-stability}
    \sup_{t\in[0,T]}\int_{\mathbb{R}^n}\rho(x,t) \log\left(\frac{\rho(x,t)}{\rho^h(x,t)}\right)\di x \le C_T\alpha_h^2.
\end{equation}
In particular, for the modified equation of Lemma \ref{lem:modified-coefficients-stability}, $\alpha_h=\mathcal{O}(h)$, so the score-ratio bound \eqref{eq:weighted-score-stability} is of order $h$ and the relative-entropy bound \eqref{eq:relative-entropy-stability} is of order $h^2$.
\end{lemma}

\begin{proof}
Consider the relative entropy
\begin{equation}
    \DKL(t) = \int_{\mathbb{R}^n}\rho \log\left(\frac{\rho}{\rho^h}\right)\di x = \int_{\mathbb{R}^n}\rho u_h \di x.
\end{equation}
Taking time derivative yields
\begin{equation}
    \frac{\di}{\di t}\DKL(t) = \int_{\mathbb{R}^n}u_h \partial_t \rho \di x + \int_{\mathbb{R}^n}\rho \partial_t u_h\di x.
\end{equation}
Substituting $\partial_t u_h $with $ \frac{\partial_t \rho}{\rho}- \frac{\partial_t \rho^h}{\rho^h}$, we get
\begin{equation}
        \frac{\di}{\di t}\DKL(t) = \int_{\mathbb{R}^n}u_h \partial_t \rho \di x - \int_{\mathbb{R}^n}\frac{\rho}{\rho^h}\partial_t \rho^h \di x +\int_{\mathbb{R}^n}\partial_t \rho \di x
\end{equation}
As $\int_{\mathbb{R}^n}\partial_t \rho \di x = \partial_t \int_{\mathbb{R}^n}\rho \di x = 0$, the third term vanishes and only the first two need be considered.

For the first term, using Fokker-Planck equation $\partial_t \rho = -\nabla\cdot(b\rho)+\nabla\nabla:(A\rho)$, it gives that
\begin{equation}
        \int u_h \partial_t \rho \di x = -\int_{\mathbb{R}^n} \rho \nabla u_h \cdot \left[ -b + \nabla\cdot A + A\nabla\log\rho \right] \di x.
\end{equation}

For the second term, using $\partial_t\rho^h = -\nabla\cdot(b_h\rho^h)+\nabla\nabla:(A_h\rho^h)$ and integrating by part gives
\begin{equation}
    \int_{\mathbb{R}^n}\frac{\rho}{\rho^h}\partial_t \rho^h\di x = \int_{\mathbb{R}^n}\nabla\left(\frac{\rho}{\rho^h}\right)\cdot\left( b_h\rho^h -\nabla\cdot A_h\rho^h -A_h\nabla\rho^h\right)\di x
\end{equation}
Since $\nabla(u_h) = \frac{\rho^h}{\rho}\nabla\left(\frac{\rho}{\rho^h}\right)$, substituting this in cancels to $\rho^h$
\begin{equation}
        -\int_{\mathbb{R}^n}\frac{\rho}{\rho^h}\partial_t \rho^h\di x  =  \int_{\mathbb{R}^n}\rho \nabla u_h\cdot[- b_h + \nabla \cdot A_h + A_h\nabla \log\rho -A_h\nabla u_h]\di x
\end{equation}
Combined together we will have
\begin{equation}
    \frac{\di}{\di t}\DKL(t)  = \int_{\mathbb{R}^n}\rho \nabla u_h\cdot \Phi_h \di x -\int_{\mathbb{R}^n}\rho\nabla u_h\cdot A_h \nabla u_h \di x , 
\end{equation}
where
\begin{equation}
    \Phi_h := -(b_h-b) + \nabla\cdot(A_h -A) + (A_h- A)\nabla \log \rho .
\end{equation}
By Lemma \ref{lem:modified-coefficients-stability}, $A_h$ is uniformly elliptic, $A_h\ge \frac{r}{2}I$, for $h \le \frac{r}{2C}$. Applying Young's inequality ($xy \le \frac{r}{4}x^2 +\frac{1}{r} y^2$), we obtain
\begin{equation}
\begin{aligned}
    &\frac{\di}{\di t}\DKL(t) + \frac{r}{2}\int_{\mathbb{R}^n}|\nabla u_h|^2\rho \di x\\ \le  &\frac{\di}{\di t}\DKL(t) +\int_{\mathbb{R}^n}\rho\nabla u_h\cdot A_h \nabla u_h \di x\\
     = &\int_{\mathbb{R}^n}\rho \nabla u_h \cdot \Phi_h \di x \le  \frac{r}{4}\int_{\mathbb{R}^n}|\nabla u_h|^2\rho \di x  + \frac{1}{r}\int_{\mathbb{R}^n}|\Phi_h|^2 \rho \di x.
    \end{aligned}
\end{equation}

Rearrange gives the energy estimate
\begin{equation}
    \frac{\di }{\di t}\DKL(t) + \frac{r}{4}\int_{\mathbb{R}^n}|\nabla u_h|^2\rho \di x \le \frac{1}{r}\int_{\mathbb{R}^n}|\Phi_h|^2 \rho \di x.
\end{equation}

By the definition \eqref{eq:fp-coefficient-discrepancy} of $\alpha_h$, the
discrepancies $b_h -b$, $A_h -A$ and $\nabla\cdot(A_h -A)$ are uniformly
bounded by $\alpha_h$, so that pointwise
\begin{equation}
    |\Phi_h(x,t)|^2 \le C\alpha_h^2\left(1 + |\nabla\log \rho(x,t) |^2\right).
\end{equation}
Integrating the energy estimate over $[0,t]$ for any $t\le T$ and invoking
the finite Fisher integral $M_T$ of Assumption
\ref{Fokker-Planck assumption} to bound
$\int_0^T\!\int_{\mathbb{R}^n}|\nabla\log\rho|^2\rho\,\di x\di s\le M_T$
yields
\begin{equation}
    \DKL(t) -\DKL(0) + \frac{r}{4}\int_0^t\int_{\mathbb{R}^n}|\nabla u_h|^2\rho \di x \di s \le \frac{C(T +M_T)\alpha_h^2}{r} .
\end{equation}
Since the two flows start from the same density, $\DKL(0)=0$; moreover
$\DKL(t)\ge0$ for every $t$. Discarding the nonnegative entropy term at
$t=T$ gives \eqref{eq:weighted-score-stability} with
$C_T^2=4C(T+M_T)/r^2$, and discarding the nonnegative dissipation term gives
\eqref{eq:relative-entropy-stability} with $C_T=C(T+M_T)/r$. In particular,
$C_T$ grows at most linearly in $T+M_T$.
\end{proof}

\begin{corollary}
\label{Tilde est}
Let the forward process satisfy Assumptions \ref{b-and-A} and
\ref{Fokker-Planck assumption} with space-independent $A(t)$, and let
$(b_h,\sigma_h)$ be the first-order modified coefficients of the
Euler--Maruyama scheme provided by Theorem \ref{thm:modified-equation}, with
density $\rho^h$ as in Lemma \ref{lem:weighted-score-stability}. Then, for
$h$ small enough, the modified process satisfies the hypotheses of
Proposition \ref{prop:time-reversal} and hence admits a time reversal that
is again a diffusion process,
\begin{equation*}
    \di X_t^{-,h} = b^-_h(X_t^{-,h},t)\di t + \sigma^-_h( X_t^{-,h},t)\di W_t,
    \quad X^{-,h}_0 \sim \rho^h(\cdot,T),
\end{equation*}
whose drift is $\mathcal{O}(h)$-close to the exact reverse drift:
\begin{equation*}
    \|b^-_h - b^-\|_{L^2_\rho(\mathbb{R}^n;[0,T])} \le C_Th.
\end{equation*}
\end{corollary}

\begin{proof}
By Assumption \ref{Fokker-Planck assumption} and Lemma
\ref{lem:weighted-score-stability},
\begin{equation}
    \| \nabla \log \rho\|_{L^2_\rho(\mathbb{R}^n;[0,T])} < + \infty,\quad \| \nabla \log \rho - \nabla\log \rho^h\|_{L^2_\rho(\mathbb{R}^n;[0,T])} < + \infty,
\end{equation}
so the triangle inequality gives
$\| \nabla \log \rho^h\|_{L^2_\rho(\mathbb{R}^n;[0,T])} < + \infty$: the
modified process satisfies the finite-Fisher-information condition, and its
diffusion matrix $A_h$ is uniformly elliptic for $h$ small by Lemma
\ref{lem:modified-coefficients-stability}, so Proposition
\ref{prop:time-reversal} applies to it. The time-reversal formula applied to
the original and modified processes gives
\begin{equation*}
    b^-(x,t)=-b(x,T-t)+2A(T-t)\nabla\log\rho(x,T-t),
\end{equation*}
and, since $A_h$ is in general state-dependent,
\begin{equation*}
    b^-_h(x,t)=-b_h(x,T-t)+2(\nabla\cdot A_h)(x,T-t)+2A_h(x,T-t)\nabla\log\rho^h(x,T-t),
\end{equation*}
where $A_h=\sigma_h(\sigma_h)^T/2$. Therefore, using $\nabla\cdot A=0$,
\begin{align*}
    b_h^--b^-
    =&-\left(b_h - b\right)(x,T-t)
    +2\big(\nabla\cdot(A_h-A)\big)(x,T-t)\\
    &+2(A_h-A)(x,T-t)\nabla\log\rho(x,T-t)\\
    &+2A_h(x,T-t)\left(\nabla\log\rho^h(x,T-t)-\nabla\log\rho(x,T-t)\right).
\end{align*}
Taking the $L^2_\rho(\mathbb{R}^n;[0,T])$ norm and using the triangle inequality yields
\begin{align*}
    \|b_h^--b^-\|_{L^2_\rho}
    \le&
    \|b_h - b\|_{L^2_\rho}
    +2\sqrt T\,\|\nabla\cdot(A_h-A)\|_{L^\infty}
    +2\|A_h-A\|_{L^\infty}\|\nabla\log\rho\|_{L^2_\rho}\\
    &+2\|A_h\|_{L^\infty}
    \|\nabla\log\rho^h-\nabla\log\rho\|_{L^2_\rho}.
\end{align*}
Lemma \ref{lem:modified-coefficients-stability} gives $\|b_h - b\|_{L^\infty}\le C h$ and $\|A_h-A\|_{L^\infty}+\|\nabla\cdot(A_h-A)\|_{L^\infty}\le Ch$. The Fisher-information assumption $\|\nabla\log\rho\|_{L^2_\rho}<\infty$ controls the third term, while Lemma \ref{lem:weighted-score-stability} gives $\|\nabla\log\rho^h-\nabla\log\rho\|_{L^2_\rho}\le C_T\alpha_h\le Ch$. Hence
\begin{equation*}
    \|b^-_h-b^-\|_{L^2_\rho(\mathbb{R}^n;[0,T])}\le C_Th.
\end{equation*}
If the constants in Lemmas \ref{lem:modified-coefficients-stability} and \ref{lem:weighted-score-stability} are uniform in time, then $C_T^2$ grows at most linearly in $T$, because the norm $L^2_\rho(\mathbb{R}^n;[0,T])$ integrates a uniformly bounded-in-time error over the interval $[0,T]$.
\end{proof}


\begin{remark}\label{rem:reverse-drift-singularity}
In general, $b^-(x,t)$ need not remain a $C^{\infty}(\mathbb{R}^n\times[0,T])$
function. For a constant diffusion matrix $\sigma \in \mathbb{R}^{n\times n}$
and data concentrated near a point $X_0$, the small-time marginal behaves
like the heat kernel,
$$\rho(x,t) \approx C\,t^{-n/2}\exp\Big(-\frac{1}{2t}(x-X_0)^T(\sigma\sigma^T)^{-1}(x- X_0)\Big),$$
so the score $\nabla\log\rho(\cdot,t)$ grows like $\mathcal O(1/t)$ as
$t\to0$, and the reverse drift $b^-(\cdot,t)$ exhibits an
$\mathcal{O}\left(\frac{1}{T-t}\right)$ singularity as $t \to T$ in the
reverse clock; see Example \ref{ex:ou-dirac} for an explicit instance.
However, one can design a noise schedule $\sigma(t)$, or rescale time, so
that $\lim\limits_{t \to 0} \frac{\sigma^2(t)}{t} < \infty$ as for the VP
and VE processes, in which case $b^-(\cdot,t)$ remains smooth even though
the score $\nabla \log \rho$ is singular.
\end{remark}

\subsection{Score matching error: relative entropy of score function}
For the score-based generative model, the reverse SDE takes the form
    \begin{equation*}
        \di X^{\theta,-}_t = [- b(X^{\theta,-}_t,T-t) + 2A(T-t) s_\theta(X^{\theta,-}_t,T-t)]\di t+ \sigma(T-t)\di W_t,\quad X^{\theta,-}_0 \sim X_T,
    \end{equation*}
with the score obtained by score matching,
\begin{equation*}
    \theta^* = \arg\min_\theta \|s_\theta(x,t)- \nabla \log \rho(x,t)\|_{L^2_\rho(\mathbb{R}^n,[0,T])}.
\end{equation*}
\begin{assumption}\label{assump:score-error}
    Suppose that
    \begin{equation*}
        \|s_{\theta^*}(x,t)- \nabla\log \rho(x,t)\|^2_{L^2_\rho(\mathbb{R}^n,[0,T])} \le \varepsilon^2,
    \end{equation*}
    where $\varepsilon$ is the training loss.
\end{assumption}
\begin{lemma}\label{lem:score-drift-error}
Under Assumption \ref{assump:score-error} and a uniform bound on $A$, the
learned reverse drift $b_{\theta^*}^-(x,t) := -b(x,T-t) + 2A(T-t) s_{\theta^*}(x,T-t)$ satisfies
\begin{equation*}
    \|b_{\theta^*}^- - b^-\|_{L^2_\rho(\mathbb{R}^n,[0,T])} \le 2\|A\|_{L^\infty}\varepsilon.
\end{equation*}
\end{lemma}
\begin{proof}
The two reverse drifts differ only through the score, since
$b^-= -b+2A\nabla\log\rho$ and $b_{\theta^*}^-=-b+2A s_{\theta^*}$ (all
reverse-clock arguments as above). Hence the drift error is the score error
amplified by the diffusion,
\begin{equation*}
    \|b^- -b_{\theta^*}^-\|_{L^2_\rho}
    \le 2\|A\|_{L^\infty}
     \|s_{\theta^*}-\nabla\log\rho\|_{L^2_\rho}
     \le 2\|A\|_{L^\infty}\varepsilon ,
\end{equation*}
which is the claimed bound.
\end{proof}

\subsection{Deterministic error}

For any forward process $X_t$, once the finite-time distribution $X_T$ is known, the initial distribution $X_0$ can be restored by the time-reversed process. In practice, the backward process is often initialized from an invariant distribution $\rho_\infty$ rather than the exact terminal law $\rho_T$, producing an initialization error.

For an ergodic SDE, there exists an invariant measure $\rho_\infty$ such that for a finite time $T$, we have an estimation that for any test function $\phi$
\begin{equation*}
    \left|\int_{\mathbb{R}^n}\phi(x)\rho_\infty(x) \di x - \int_{\mathbb{R}^n} \phi(x)\rho(x,T)\di x\right| \le C e^{-\kappa T}.
\end{equation*}

\begin{proposition}\label{prop:lsi-decay}
    For an ergodic SDE that converges to an invariant measure $\rho_\infty$, if $\rho_\infty$ satisfies the logarithmic Sobolev inequality ($LSI$) \cite{bakry2014logarithmic}, then there exists the logarithmic Sobolev constant $C_{LS}$ such that
    \begin{equation*}
        \DKL(\rho_t||\rho_\infty) \le e^{\frac{-2t}{C_{LS}}}\DKL(\mu_0||\rho_\infty).
    \end{equation*}
    If $\rho_\infty$ is a standard Gaussian distribution, then $C_{LS}=1$.
\end{proposition}

In the backward process, this small error will propagate within time $T$.

\subsection{Error analysis for score-based generative model}
\label{sec:overall-error}

The three contributions estimated above discretization now can be combined through the endpoint identity into a single bound on the generative error. The
entropy-production identity is applied once, to the pair formed by the
forward process and the modified backward process of the sampler, and the
three error sources appear as separate parts of the resulting forcing term.

\begin{theorem}[Total error of the discretized score-based sampler]\label{thm:total-error}
Let the forward process satisfy Assumptions \ref{b-and-A} and
\ref{Fokker-Planck assumption} with a space-independent diffusion matrix
$A(t)$, and let the forward process be ergodic with Gaussian invariant measure
$\rho_\infty=\mathcal{N}(0,I)$ satisfying the logarithmic Sobolev inequality
of Proposition \ref{prop:lsi-decay}. Let $X^{\theta^*,h,-}_t$ denote the
modified backward process associated, through Theorem
\ref{thm:modified-equation}, with the Euler--Maruyama discretization of the
learned reverse SDE,
    \begin{equation*}
    \di X^{\theta^*,h,-}_t =b_{\theta^*,h}^-(X^{\theta^*,h, -}_t,t)\di t + \sigma_h(X^{\theta^*,h,-}_t,t)\di W_t, \quad X^{\theta^*,h,-}_0 \sim \rho_\infty = \mathcal{N}(0,I),
\end{equation*}
where $b_{\theta^*,h}^-$ and $A_h=\tfrac12\sigma_h\sigma_h^T$ are the
first-order modified coefficients, and let $h$ be small enough that
$A_h\ge\tfrac{r}{2}I$ (Lemma \ref{lem:modified-coefficients-stability}). Then
   \begin{equation}\label{eq:total-error-bound}
       \DKL(X_0|| X^{\theta^*,h,-}_T) \le C_1(T)h^2 + C_2(T)\varepsilon^2 + e^{-\frac{2T}{C_{LS}}}\DKL(\mu_0||\rho_\infty),
   \end{equation}
where $C_2(T)=8\|A\|^2_{L^\infty}/r$ and $C_1(T)\le C(T+M_T)/r$, with $M_T$
the Fisher integral of \eqref{fisher information} and $C$ determined by the
coefficient bounds of Lemma \ref{lem:modified-coefficients-stability}; in
particular, both constants grow at most linearly in $T$ when the coefficient
bounds are uniform in time.
\end{theorem}
\begin{proof}
Let $\rho'(\cdot,s)$ denote the density of $X^{\theta^*,h,-}_s$, so that
$\rho'(\cdot,0)=\rho_\infty$, and set
$u(x,t):=\log\frac{\rho(x,t)}{\rho'(x,T-t)}$; as in Section
\ref{sec:Entropy}, primed quantities are evaluated at reverse time $T-t$. In
the smooth, uniformly elliptic setting with finite Fisher information, the
regularity Assumption \ref{assump:ep-regularity} holds for the pair
$(\rho,\rho')$, so the entropy-production identity applies. The
Euler--Maruyama iterates themselves agree with the marginals of the modified
flow $\rho'$ to weak order $\mathcal{O}(h^2)$ by Theorem
\ref{thm:modified-equation}; the theorem bounds the error of the modified
flow, which is the continuous object that the scheme tracks at this order.

With
$D(t):=\DKL(\rho(\cdot,t)\,\|\,\rho'(\cdot,T-t))$, the endpoint identity
\eqref{eq:KL-divergence_equality} reads
\begin{equation}\label{eq:proof-bookkeeping}
    \DKL(X_0\,\|\,X^{\theta^*,h,-}_T)=D(0)=D(T)+\int_0^T e_p(t)\,\di t,
    \qquad D(T)=\DKL(\rho_T\,\|\,\rho_\infty),
\end{equation}
and the initialization term decays exponentially by Proposition
\ref{prop:lsi-decay}:
$\DKL(\rho_T \,\|\, \rho_\infty) \le e^{-2T/C_{LSI}}\DKL(\mu_0\,\|\,\rho_\infty)$.

It remains to bound the integrated entropy production, and for this we return
to the velocity form of the identity. The marginals
$\rho(\cdot,t)$ and $\rho'(\cdot,T-t)$ satisfy the two continuity equations
of Section \ref{sec:continuity} with current velocities $v_c$ and $v_c'$ of
\eqref{eq:fwd-velocity}--\eqref{eq:bwd-velocity}, the latter with
$b'=b^-_{\theta^*,h}$ and $A'=A_h$. By the identity established in the proof
of Theorem \ref{thm:mother-identity},
\begin{equation}\label{eq:proof-energy-identity}
    \frac{\di}{\di t}D(t)=\int_{\mathbb{R}^n}\rho\,\nabla u\cdot(v_c+v_c')\,\di x .
\end{equation}
We split $v_c+v_c'$ into a dissipative part and a forcing part. Using
$b^-_{\theta^*}(x,T-t)=-b(x,t)+2A(t)s_{\theta^*}(x,t)$, the identity
$A_h\nabla\log\rho'=A_h\nabla\log\rho-A_h\nabla u$, and $\nabla\cdot A=0$
for the space-independent forward diffusion,
\begin{equation*}
    v_c+v_c'
    =\big[b-A\nabla\log\rho\big]
    +\big[b^-_{\theta^*,h}(\cdot,T-t)-\nabla\cdot A_h-A_h\nabla\log\rho'\big]
    =A_h\nabla u+\Phi_h,
\end{equation*}
where the forcing collects the score and discretization discrepancies,
\begin{equation}\label{eq:Phi-def}
    \Phi_h:=2A\big(s_{\theta^*}-\nabla\log\rho\big)
    +\big(b^-_{\theta^*,h}-b^-_{\theta^*}\big)(\cdot,T-t)
    +(A-A_h)\nabla\log\rho-\nabla\cdot(A_h-A),
\end{equation}
using $b+b^-_{\theta^*}(\cdot,T-t)=2As_{\theta^*}$. We write
$\Phi_{\text{score}}:=2A(s_{\theta^*}-\nabla\log\rho)$ for the first term of
\eqref{eq:Phi-def} and $\Phi_{\text{disc}}$ for the sum of the remaining
three, so that $\Phi_h=\Phi_{\text{score}}+\Phi_{\text{disc}}$ separates the
score-approximation from the discretization discrepancy.

Substituting this splitting into
\eqref{eq:proof-energy-identity} and using the ellipticity
$A_h\ge \tfrac{r}{2}I$ together with Young's inequality
$ab\le \tfrac{r}{4} a^2+\tfrac{1}{r} b^2$,
\begin{equation*}
    \frac{\di}{\di t}D(t)
    \ \ge\ \frac{r}{2}\|\nabla u\|^2_{L^2_\rho}
    -\|\nabla u\|_{L^2_\rho}\|\Phi_h\|_{L^2_\rho}
    \ \ge\ \frac{r}{4}\|\nabla u\|^2_{L^2_\rho}-\frac{1}{r}\|\Phi_h\|^2_{L^2_\rho},
\end{equation*}
so that $e_p(t)=-\tfrac{\di}{\di t}D(t)\le \tfrac{1}{r}\|\Phi_h(\cdot,t)\|^2_{L^2_\rho}$
and, after integration and the splitting $(x+y)^2\le 2x^2+2y^2$,
\begin{equation}\label{eq:total_ep_bound}
    \int_0^T e_p(t) \di t \le \frac{1}{r} \int_0^T \int_{\mathbb{R}^n} \rho |\Phi_h|^2 \di x \di t
    \le \frac{2}{r}\int_0^T\Big(\|\Phi_{\text{score}}\|^2_{L^2_\rho}
    +\|\Phi_{\text{disc}}\|^2_{L^2_\rho}\Big)\di t .
\end{equation}

The two contributions are now estimated in turn. For the score part, by
Assumption \ref{assump:score-error}
(equivalently, Lemma \ref{lem:score-drift-error}),
\begin{equation*}
    \frac{2}{r}\int_0^T\|\Phi_{\text{score}}\|^2_{L^2_\rho}\di t
    =\frac{8}{r}\int_0^T\big\|A(s_{\theta^*}-\nabla\log\rho)\big\|^2_{L^2_\rho}\di t
    \le \frac{8\|A\|^2_{L^\infty}}{r}\,\varepsilon^2=C_2(T)\,\varepsilon^2 .
\end{equation*}

For the discretization part, applying Theorem
\ref{thm:modified-equation} and Lemma
\ref{lem:modified-coefficients-stability} to the reverse SDE with drift
$b^-_{\theta^*}$ (which lies in $C_b^{k}$ because $b$ and $s_{\theta^*}$ do),
\begin{equation*}
    \|b^-_{\theta^*,h}-b^-_{\theta^*}\|_{L^\infty}\le Ch,\qquad
    \|A_h-A\|_{L^\infty}+\|\nabla\cdot(A_h-A)\|_{L^\infty}\le Ch .
\end{equation*}
Hence, bounding the middle term of $\Phi_{\text{disc}}$ through the Fisher
integral $M_T$ of Assumption \ref{Fokker-Planck assumption},
\begin{equation*}
    \frac{2}{r}\int_0^T\|\Phi_{\text{disc}}\|^2_{L^2_\rho}\di t
    \le \frac{6C^2}{r}\,h^2\big(2T+M_T\big)=:C_1(T)\,h^2 .
\end{equation*}

Combining the initialization, score, and discretization estimates through the identity \eqref{eq:proof-bookkeeping} yields
\eqref{eq:total-error-bound}. Under coefficient bounds uniform in time,
$C_2$ does not depend on $T$ (the time dependence being absorbed in the
definition of $\varepsilon$), while $C_1(T)$ grows linearly in $T+M_T$.
\end{proof}

\begin{remark}
\label{rem:marginal-vs-path}
The proof accumulates the entropy production at the level of the
\emph{marginals}: the dissipative term $A_h\nabla u$ is retained and only the
coefficient discrepancies enter the forcing, which is why the discretization
contribution appears at order $h^2$ (the square of the $\mathcal{O}(h)$
coefficient discrepancy). A Girsanov-based argument bounds instead the
path-space divergence, which dominates the marginal one and produces an
$\mathcal{O}(h)$ discretization term; the marginal accounting removes exactly
this path-to-marginal slack. The rate $h^2$ is directly testable, and the
experiments of Section \ref{sec:Numerical} display marginal-KL slopes close
to $2$ in $\log h$.
\end{remark}

\begin{lemma}\label{lem:mollification-gap}
Under Assumption \ref{assump:second-moment}, for every $\delta\in(0,T)$,
\begin{equation}\label{eq:mollification-gap}
    W_2(\rho_0,\rho_\delta)\ \le\ (1-\alpha_\delta)\sqrt{m_2}\ +\ \beta_\delta\sqrt n .
\end{equation}
In particular, for the OU schedule $\alpha_\delta=e^{-\delta/2}$,
$\beta_\delta=\sqrt{1-e^{-\delta}}$, the right-hand side is
$\mathcal O\!\big(\sqrt{\delta\,(m_2+n)}\big)$ as $\delta\to0$.
\end{lemma}
 
\begin{proof}
We couple $\rho_0$ and $\rho_\delta$ through the pair
$(X_0,\ \alpha_\delta X_0+\beta_\delta\epsilon)$,
which is admissible for the Wasserstein distance. Then
\[
    W_2^2(\rho_0,\rho_\delta)\ \le\
    \mathbb E\big|(1-\alpha_\delta)X_0-\beta_\delta\epsilon\big|^2
    \ =\ (1-\alpha_\delta)^2\,m_2+\beta_\delta^2\,n,
\]
since the cross term vanishes by independence and $\mathbb E\epsilon=0$.
Taking square roots and using $\sqrt{a+b}\le\sqrt a+\sqrt b$ gives
\eqref{eq:mollification-gap}. For OU, $1-\alpha_\delta\sim\delta/2$ and
$\beta_\delta\sim\sqrt\delta$, so the bound is $\mathcal O(\sqrt\delta)$.
\end{proof}
 
We can now state the early-stopped, weakened counterpart of
Theorem \ref{thm:total-error}.
 
\begin{corollary}
\label{cor:early-stopped}
Consider the affine forward process \eqref{eq:affine-marginal} with the
structural smoothness and uniform ellipticity of Assumption \ref{b-and-A},
and let the learned reverse drift be initialized from
$\rho_\infty=\mathcal N(0,I)$ and discretized by the Euler--Maruyama scheme
with step $h$. Replace Assumption \ref{Fokker-Planck assumption} by the
finite-second-moment Assumption \ref{assump:second-moment}, and assume the
score-error bound of Assumption \ref{assump:score-error} on the truncated
interval, i.e.
$\|s_{\theta^*}-\nabla\log\rho\|^2_{L^2_\rho(\mathbb R^n\times[\delta,T])}\le\varepsilon^2$.
Then for every early-stopping level $\delta\in(0,T)$ the early-stopped
generated law $\rho^{\theta^*,h,-}_{T-\delta}$ (the sampler output at
reverse-clock time $T-\delta$) satisfies
\begin{equation}\label{eq:early-stopped-KL}
    \DKL\!\big(\rho_\delta\,\big\|\,\rho^{\theta^*,h,-}_{T-\delta}\big)
    \ \le\
    C_1(T,\delta)\,h^2+C_2(T,\delta)\,\varepsilon^2
    +e^{-2T/C_{LS}}\,\DKL(\rho_\delta\,\|\,\rho_\infty),
\end{equation}
where the constants are finite for every $\delta>0$. Combining with
Lemma \ref{lem:mollification-gap} and the triangle inequality, the error to
the \emph{true} (possibly singular) data distribution is controlled in
Wasserstein distance:
\begin{equation}\label{eq:early-stopped-W2}
    W_2\!\big(\rho_0,\rho^{\theta^*,h,-}_{T-\delta}\big)
    \ \le\
    W_2\!\big(\rho_\delta,\rho^{\theta^*,h,-}_{T-\delta}\big)
    +(1-\alpha_\delta)\sqrt{m_2}+\beta_\delta\sqrt n .
\end{equation}
\end{corollary}
 
\begin{proof}
The argument is that of Theorem \ref{thm:total-error}, run on $[\delta,T]$
instead of $[0,T]$. By the identity
\eqref{eq:KL-divergence_equality} applied between forward times $\delta$ and
$T$,
\[
    \DKL\!\big(\rho_\delta\,\|\,\rho^{\theta^*,h,-}_{T-\delta}\big)
    =\DKL(\rho_T\,\|\,\rho_\infty)+\int_\delta^T e_p(t)\,\di t .
\]
The initialization term is bounded by
$e^{-2T/C}\DKL(\mu_0\|\rho_\infty)$ exactly as in
Proposition \ref{prop:lsi-decay}, unchanged by truncation. For the integrated
entropy production, every step of the proof of Theorem \ref{thm:total-error}
that invoked the global Fisher integral now invokes its truncated counterpart
$M_{\delta,T}$, which Proposition \ref{prop:malliavin-cap} bounds by $nJ_\delta<\infty$
under Assumption \ref{assump:second-moment} alone; the marginals
$\rho_t$ are smooth and positive on $[\delta,T]$ by the same lemma, so the
energy estimate of Lemma \ref{lem:weighted-score-stability} and the
score/discretization split apply verbatim with all time integrals restricted
to $[\delta,T]$. This yields \eqref{eq:early-stopped-KL}. The Wasserstein
bound \eqref{eq:early-stopped-W2} follows from the triangle inequality and
Lemma \ref{lem:mollification-gap}.
\end{proof}

\section{A unified view of deterministic samplers, flow matching, and
stochastic interpolants}
\label{sec:unified-comparison}

The velocity form of the entropy production identity
(Theorem \ref{thm:mother-identity}) is parametrized by the diffusion of the
forward and comparison processes. Since probability-flow (PF) sampling, flow
matching, and stochastic interpolants \cite{albergo2023stochastic} differ from
the score-based SDE precisely in this choice, the identity provides a common
language in which to compare them. We proceed in two steps. First we derive a
\emph{uniform} error analysis in which every framework is controlled by the
same two quantities, the current-velocity error $\eta$ and the score-matching
error $\varepsilon$ (Theorem \ref{thm:unified-error}). Then we minimize the
resulting bound over the sampler diffusion, obtaining the optimal noise level
for the general SDE sampler and for the stochastic interpolant
(Corollaries \ref{cor:optimal-diffusion} and \ref{cor:interpolant-optimal}).

\subsection{The adjustable family and the two error channels}
\label{sec:adjustable-family}

The forward current velocity \eqref{eq:fwd-velocity},
$v_{c}=b-\nabla\cdot A-A\nabla\log\rho$, is determined by the
marginal family $\{\rho_t\}$ alone and does not depend on the sampler. Every
sampler that reproduces $\{\rho_t\}$ therefore shares the same
$v_{c}$ and differs only through its diffusion. We keep the
forward noise coefficient $\sigma(t)$ with
$A:=\tfrac12\sigma\sigma^{T}$, and parametrize the sampler by its diffusion
coefficient $A_s(t)\ge0$, i.e.\ reverse noise $\sigma_s = (2A_s)^\frac{1}{2}$, as in
the stochastic interpolant framework; for clarity we treat $A_s$ (and $A$)
as scalar. The Anderson adjustable-noise family takes
\begin{equation}\label{eq:adjustable-noise}
    \di X^s_t=\big[-b+(A+A_s)\nabla\log\rho\big]\di t
    +(2A_s)^\frac{1}{2}\,\di W_t,
\end{equation}
all evaluated at $T-t$:
\begin{enumerate}
\item $A_s=0$ gives the PF-ODE, whose velocity takes $-b+A\nabla\log\rho$;
\item $A_s=A$, i.e.\ reverse noise $\sigma_s=\sigma$ equal to the
forward noise, gives the standard reverse SDE drift
$-b+2A\nabla\log\rho$.
\end{enumerate}
Table \ref{tab:frameworks} records the resulting dictionary.

\begin{table}[t]
\centering
\small
\setlength{\tabcolsep}{4.5pt}
\begin{tabular}{lccl}
\toprule
Framework & sampler diffusion $\sigma_s$ & learned object\\
\midrule
Score SDE / DDPM                 & $\sigma_s= \sigma$  & score $\nabla\log\rho$\\
PF-ODE / deterministic          & $\sigma_s=0$                  & score $\nabla\log\rho$\\
Flow matching             & $\sigma_s=0$                  & current velocity $v_c$\\
Stochastic interpolant   & $\sigma_s$ free            & current velocity $v_c$ /score $\nabla \log \rho$\\
\bottomrule
\end{tabular}
\caption{The frameworks as choices of forward noise $\sigma$ and sampler
diffusion $\sigma_s$ in \eqref{eq:adjustable-noise}, with
$A=\tfrac12\sigma\sigma^{T}$. The smoothing kernel $\gamma_t z$ of a stochastic
interpolant is recorded as a fixed forward diffusion that defines the
marginals and the score; the noise level $\sigma_s$ is a separate sampler
diffusion, decoupled from the smoothing.}
\label{tab:frameworks}
\end{table}

Since the reverse drift of \eqref{eq:adjustable-noise} can be written through
the current velocity as $-v_c+A_s\nabla\log\rho$ (evaluated at $T-t$), a
practical sampler replaces the two exact fields by learned ones: a learned
current velocity $\hat v$ and/or a learned score $\hat s$ (one of the two may
be absent, per Table \ref{tab:frameworks}),
\begin{equation}\label{eq:learned-sampler}
    \di \hat X^s_t=\big[-\hat v+A_s\,\hat s\big](\hat X^s_t,T-t)\,\di t
    +(2A_s)^\frac{1}{2}\,\di W_t .
\end{equation}
All error bounds of this section are expressed through the two resulting
error channels,
\begin{equation}\label{eq:error-notation}
    \eta_t:=\|\hat v-v_c\|_{L^2_{\rho_t}},\qquad
    \varepsilon_t:=\|\hat s-\nabla\log\rho_t\|_{L^2_{\rho_t}},\qquad
    \eta^2:=\int_\delta^T\eta_t^2\,\di t,\qquad
    \varepsilon^2:=\int_\delta^T\varepsilon_t^2\,\di t,
\end{equation}
the current-velocity error and the score-matching error; the
latter is Assumption \ref{assump:score-error} restricted to $[\delta,T]$.

Both channels are exactly the quantities that training controls. For the
score channel, the denoising objective \eqref{eq:epsilon_prediction} equals
$\varepsilon_t^2$ plus an irreducible, $\theta$-independent floor. For the
velocity channel, flow matching and the stochastic interpolant learn $\hat v$
by regressing the interpolation velocity $\dot x_t$ on $x_t$, and the
regression loss decomposes as
\begin{equation}\label{eq:fm-loss-decomposition}
    \mathbb E\,\big\|u_\theta(x_t,t)-\dot x_t\big\|^2
    =\big\|u_\theta-\mathbb E[\dot x_t\mid x_t=\cdot\,]\big\|^2_{L^2_{\rho_t}}
    +\mathbb E\,\big\|\dot x_t-\mathbb E[\dot x_t\mid x_t]\big\|^2 ,
\end{equation}
where the second term is an irreducible conditional-variance floor
independent of $\theta$. The regression target
$\mathbb E[\dot x_t\mid x_t=x]$ is precisely the current velocity of the
marginal flow where Lemma \ref{lem:velocity-score} below computes it in closed
form, and \eqref{eq:c-equals-A} identifies it with $v_c$; in the generation
clock the target is $-v_c(\cdot,T-t)$, so the magnitude $\eta_t$ is
independent of the direction of time. Hence $\eta_t^2$ is the trainable
excess of the flow-matching loss above its floor: the exact analogue, for
the velocity channel, of what $\varepsilon_t^2$ is for the score channel.

Writing $\delta v:=\hat v-v_c$ and $\delta s:=\hat s-\nabla\log\rho$, the
drift error of \eqref{eq:learned-sampler} is $\delta c=-\delta v+A_s\,\delta s$,
so that
\begin{equation}\label{eq:drift-error-split}
    \|\delta c(\cdot,t)\|_{L^2_\rho}\ \le\ \eta_t+A_s\,\varepsilon_t .
\end{equation}

The two channels are not always independent: when the velocity is itself
built from a score, $\eta_t$ and $\varepsilon_t$ are tied pointwise, with a
ratio fixed by the parametrization. The following lemma identifies this
ratio for a Gaussian interpolant schedule.

\begin{lemma}
\label{lem:velocity-score}
Let $x_t=\alpha_t x_1+\beta_t x_0$ with $x_1\sim\rho_*$ and $x_0\sim\mathcal N(0,I)$
independent, $\alpha_t,\beta_t>0$ for $t\in(0,1)$, and let
$\rho_t=\operatorname{law}(x_t)$, $s=\nabla\log\rho_t$. Then the
probability-flow velocity $v(x,t)=\mathbb E[\dot x_t\mid x_t=x]$ satisfies
\begin{equation}\label{eq:velocity-score}
    v(x,t)=\frac{\dot\alpha_t}{\alpha_t}\,x+c(t)\,s(x,t),
    \qquad
    c(t):=\frac{\dot\alpha_t\beta_t^2}{\alpha_t}-\dot\beta_t\beta_t
    =\beta_t^2\,\frac{\di}{\di t}\log\frac{\alpha_t}{\beta_t}.
\end{equation}
In the generation direction the signal-to-noise ratio $\alpha_t/\beta_t$
increases, so $c(t)>0$.
\end{lemma}

\begin{proof}
Conditionally on $x_1$, $x_t\sim\mathcal N(\alpha_t x_1,\beta_t^2 I)$, so by
Tweedie's formula
$s(x,t)=\mathbb E[\nabla_x\log p(x\mid x_1)\mid x_t=x]
=-\beta_t^{-1}\,\mathbb E[x_0\mid x_t=x]$,
hence $\mathbb E[x_0\mid x_t=x]=-\beta_t s$ and, from
$x=\alpha_t\mathbb E[x_1\mid x_t=x]+\beta_t\mathbb E[x_0\mid x_t=x]$,
$\mathbb E[x_1\mid x_t=x]=\alpha_t^{-1}(x+\beta_t^2 s)$. Therefore
\[
    v=\dot\alpha_t\,\mathbb E[x_1\mid x_t=x]+\dot\beta_t\,\mathbb E[x_0\mid x_t=x]
    =\frac{\dot\alpha_t}{\alpha_t}x+\Big(\frac{\dot\alpha_t\beta_t^2}{\alpha_t}-\dot\beta_t\beta_t\Big)s,
\]
which is \eqref{eq:velocity-score}; the last form uses
$\tfrac{\dot\alpha}{\alpha}-\tfrac{\dot\beta}{\beta}=\tfrac{\di}{\di t}\log\tfrac{\alpha}{\beta}$.
\end{proof}

Three consequences fix the ratio $\eta_t/\varepsilon_t$ for the frameworks of
Table \ref{tab:frameworks}.
\begin{enumerate}
\item[(i)] Score parametrization. If the sampler's velocity is formed
from the learned score, $\hat v=b-\nabla\cdot A-A\hat s$ as in
score based diffusion and the PF-ODE, then $\delta v=-A\,\delta s$ and
\begin{equation}\label{eq:eta-eps-score}
    \eta_t=A(t)\,\varepsilon_t ;
\end{equation}
moreover the two contributions to $\delta c$ are aligned, so
$\|\delta c\|_{L^2_\rho}=(A+A_s)\,\varepsilon_t$ exactly.
\item[(ii)] Interpolant schedule. If the velocity is formed from the
score through \eqref{eq:velocity-score}, then $\hat v-v=c(t)\,\delta s$ and
\begin{equation}\label{eq:eta-eps}
    \eta_t=c(t)\,\varepsilon_t ,
\end{equation}
again with aligned contributions, so
$\|\delta c\|_{L^2_\rho}=(c(t)+A_s)\,\varepsilon_t$.
\item[(iii)] Velocity parametrization. Flow matching learns $\hat v$
directly, by the regression \eqref{eq:fm-loss-decomposition}, and uses no
score; only the channel $\eta_t$ is present.
\end{enumerate}
For any interpolant whose marginals coincide with those of an affine forward
diffusion $\di x=f(t)x\,\di t+g(t)\,\di W_t$ that an admissible schedule
in the sense of Proposition \ref{prop:admissible-schedule}, with
$\alpha_t=e^{\int_0^t f}$ and
$\beta_t^2=\alpha_t^2\!\int_0^t g_s^2\alpha_s^{-2}\di s$, the schedule
coefficient reduces exactly to the forward diffusion,
\begin{equation}\label{eq:c-equals-A}
    c(t)=\tfrac12 g^2(t)=A(t),
\end{equation}
so that case (ii) contains case (i) as the special schedule $c=A$. Indeed
$\dot\alpha/\alpha=f$ and $\dot{(\beta^2)}=2f\beta^2+g^2$ give
$\dot\beta/\beta=f+\tfrac{g^2}{2\beta^2}$, hence
$c=\beta^2(\dot\alpha/\alpha-\dot\beta/\beta)=-\tfrac12g^2$ in forward time,
i.e.\ $+\tfrac12g^2$ after the noise$\to$data reversal. Substituting these
two relations into \eqref{eq:velocity-score} (in the forward clock) gives
\begin{equation*}
    \mathbb E\big[\dot X_t\,\big|\,X_t=x\big]
    =f(t)\,x-A(t)\,\nabla\log\rho(x,t)
    =b(x,t)-A(t)\,\nabla\log\rho(x,t)=v_c(x,t):
\end{equation*}
the regression target of flow matching is exactly the current
velocity of the marginal flow, which makes the identification of the learned
object with $v_c$ in Table \ref{tab:frameworks}, and the error channel
$\eta_t$ of \eqref{eq:error-notation}, canonical rather than a modeling
choice.

\subsection{Uniform error estimates for the frameworks}
\label{sec:uniform-estimates}

\begin{proposition}\label{prop:adjustable-bound}
Let the forward noise be $\sigma(t)$ with $A=\tfrac12\sigma\sigma^{T}$ and
marginals $\{\rho_t\}$, and let the sampler \eqref{eq:learned-sampler} use
diffusion $A_s>0$ with the error channels \eqref{eq:error-notation}. Suppose
the sampler marginals satisfy a logarithmic Sobolev inequality with constant
$C_{\mathrm{LS}}$. Then, writing $u=\log(\rho/\hat\rho)$,
\begin{equation}\label{eq:adjustable-diffineq}
    \frac{\di}{\di t}\DKL(\rho\,\|\,\hat\rho)
    \ \le\ -\frac{A_s}{C_{\mathrm{LS}}}\,\DKL(\rho\,\|\,\hat\rho)
    +\frac{1}{2A_s}\,\big(\eta_t+A_s\,\varepsilon_t\big)^2.
\end{equation}
Consequently, with $T_\delta:=T-\delta$,
\begin{equation}\label{eq:adjustable-integrated}
    \DKL\big(\rho_\delta\,\|\,\hat\rho_\delta\big)
    \ \le\
    e^{-A_sT_\delta/C_{\mathrm{LS}}}\,\DKL(\rho_T\,\|\,\rho_\infty)
    +\frac{1}{2A_s}\int_\delta^T\!\big(\eta_t+A_s\,\varepsilon_t\big)^2\,\di t,
\end{equation}
and adding the discretization contribution of Theorem \ref{thm:total-error}
gives the three-term bound
\begin{equation}\label{eq:three-term}
    \DKL\big(\rho_\delta\,\|\,\hat\rho_\delta\big)
    \ \lesssim\
    C_{\mathrm{disc}}(A_s)\,h^{2}
    +\frac{1}{2A_s}\int_\delta^T\!\big(\eta_t+A_s\,\varepsilon_t\big)^2\,\di t
    +e^{-A_sT_\delta/C_{\mathrm{LS}}}\,\DKL(\rho_T\,\|\,\rho_\infty).
\end{equation}
For a time dependent diffusion $A_s(t)$, the exponent $A_sT_\delta$ is
replaced by $\int_\delta^T A_s(t)\,\di t$ and the coefficients enter
pointwise in time.
\end{proposition}

\begin{proof}
The relative entropy dissipation for two Fokker--Planck flows with common
diffusion $A_s$ and drifts differing by $\delta c$ reads
$\tfrac{\di}{\di t}\DKL(\rho\|\hat\rho)
=-A_s\!\int\rho|\nabla u|^2\,\di x+\int\rho\,\nabla u\cdot\delta c\,\di x$;
this is the explicit diffusion form of the computation in the proof of
Theorem \ref{thm:mother-identity}. For \eqref{eq:learned-sampler} the drift
error is $\delta c=-\delta v+A_s\,\delta s$, bounded in $L^2_\rho$ by
\eqref{eq:drift-error-split}. Young's inequality gives
$\int\rho\,\nabla u\cdot\delta c\le\tfrac{A_s}{2}\int\rho|\nabla u|^2
+\tfrac{1}{2A_s}\int\rho|\delta c|^2$, so
$\tfrac{\di}{\di t}\DKL\le-\tfrac{A_s}{2}\int\rho|\nabla u|^2
+\tfrac{1}{2A_s}\|\delta c\|^2_{L^2_\rho}$; the logarithmic Sobolev
inequality, $\int\rho|\nabla u|^2\ge\tfrac{2}{C_{\mathrm{LS}}}\DKL$, yields
\eqref{eq:adjustable-diffineq}. Integrating the linear differential inequality
and bounding the initialization term by Proposition \ref{prop:lsi-decay} gives
\eqref{eq:adjustable-integrated}; the discretization term is appended from
Theorem \ref{thm:total-error}. Deterministic samplers admit higher-order
integrators, which is recorded in the dependence of $C_{\mathrm{disc}}$ on
$A_s$.
\end{proof}

\begin{remark}
\label{rem:deterministic-limit}
At $A_s=0$ the dissipation in \eqref{eq:adjustable-diffineq} vanishes and
the bound degenerates, consistent with a deterministic flow neither amplifying
nor contracting relative entropy. If $\Phi_{s}$ denotes the exact PF-ODE flow,
Liouville's theorem gives
$\DKL(\Phi_{s\#}\mu\,\|\,\Phi_{s\#}\nu)=\DKL(\mu\,\|\,\nu)$: an initialization
error is transported unchanged, and a velocity error accumulates only through
Gr\"onwall's inequality with the flow's expansion rate (the logarithmic norm
of the velocity Jacobian), with no damping. The $A_s=0$ analysis
therefore requires a Lipschitz bound on the learned velocity and yields a
Gr\"onwall constant $e^{LT_\delta}$ in place of the contraction
$e^{-A_sT_\delta/C_{\mathrm{LS}}}$ available for $A_s>0$: deterministic
samplers do not self-correct. For this reason the deterministic frameworks
are controlled below in the Wasserstein metric, which for singular targets
parallels the early-stopping downgrade of Section \ref{sec:early-stopping}.
Finally, for a stochastic interpolant connecting two general endpoint
distributions $\rho_0,\rho_1$, the initialization term
$e^{-A_sT_\delta/C_{\mathrm{LS}}}\DKL(\rho_T\|\rho_\infty)$ in
\eqref{eq:three-term} is replaced by the corresponding
$\DKL(\rho_1\|\hat\rho_1)$ contribution.
\end{remark}

We now instantiate the bound for each framework of
Table \ref{tab:frameworks}, using the dictionary
\eqref{eq:eta-eps-score}--\eqref{eq:eta-eps} to express every estimate
through the two channels $(\eta,\varepsilon)$ alone. We write the bounds in
the amplitude notation $A=\tfrac12\sigma^2$, $A_s=\tfrac12\sigma_s^2$, with
$\underline\sigma\le\sigma(t)\le\bar\sigma$ on $[\delta,T]$. For the
score-parametrized samplers the aligned drift error
$(A+A_s)\varepsilon_t$ of case (i) makes the score-error coefficient
\begin{equation}\label{eq:Psi-sigma}
    \Psi(\sigma_s)=\frac{(A+A_s)^2}{2A_s}
    =\frac{\big(\sigma^2+\sigma_s^2\big)^2}{4\,\sigma_s^2},
    \qquad
    \min_{\sigma_s>0}\Psi=\Psi(\sigma_s=\sigma)=\sigma^2 ,
\end{equation}
and the contraction rate is $A_s/C_{\mathrm{LS}}=\sigma_s^2/(2C_{\mathrm{LS}})$.
Stochastic samplers ($\sigma_s>0$) are therefore controlled in relative
entropy with geometric forgetting, while deterministic samplers
($\sigma_s=0$) lose the dissipation and are controlled in Wasserstein
distance through the transport mechanism of
Remark \ref{rem:deterministic-limit}.

\begin{theorem}[Uniform error estimates]\label{thm:unified-error}
Under the hypotheses of Proposition \ref{prop:adjustable-bound} (for the
stochastic samplers) and a one-sided Lipschitz bound $L$ on the learned
reverse velocity (for the deterministic samplers), with $T_\delta=T-\delta$:
\begin{enumerate}
\item[\textnormal{(a)}] Score SDE / DDPM $(\sigma_s=\sigma$, score
parametrization$)$:
\begin{equation}\label{eq:est-scoresde}
    \DKL\big(\rho_\delta\,\|\,\hat\rho_\delta\big)
    \ \le\
    C_{\mathrm{disc}}\,h^{2}
    +\bar\sigma^{2}\,\varepsilon^{2}
    +e^{-\underline\sigma^{2}T_\delta/(2C_{\mathrm{LS}})}\,\DKL(\rho_T\,\|\,\rho_\infty).
\end{equation}

\item[\textnormal{(b)}] General stochastic sampler (noise level
$\sigma_s>0$, score parametrization$)$:
\begin{equation}\label{eq:est-generalsde}
    \DKL\big(\rho_\delta\,\|\,\hat\rho_\delta\big)
    \ \le\
    C_{\mathrm{disc}}(\sigma_s)\,h^{2}
    +\frac{\big(\bar\sigma^{2}+\sigma_s^{2}\big)^{2}}{4\,\sigma_s^{2}}\,\varepsilon^{2}
    +e^{-\sigma_s^{2}T_\delta/(2C_{\mathrm{LS}})}\,\DKL(\rho_T\,\|\,\rho_\infty).
\end{equation}
The score-error coefficient decreases as $\sigma_s\uparrow\sigma$ and increases
for $\sigma_s>\sigma$; the forgetting rate increases monotonically in
$\sigma_s$; case \textnormal{(a)} is the configuration $\sigma_s=\sigma$
minimizing \eqref{eq:Psi-sigma}.

\item[\textnormal{(c)}] Probability-flow ODE $(\sigma_s=0$, score
parametrization$)$: 
the score enters the deterministic velocity with weight
$A=\tfrac12\sigma^{2}$, i.e.\ $\eta_t=A\varepsilon_t$, and
\begin{equation}\label{eq:est-pfode}
    W_2\big(\rho_\delta,\hat\rho_\delta\big)
    \ \le\
    e^{LT_\delta}\Big[\,W_2(\rho_T,\rho_\infty)+\tfrac12\,\bar\sigma^{2}\sqrt{T_\delta}\,\varepsilon\,\Big].
\end{equation}
The relative-entropy bound \eqref{eq:est-generalsde} diverges as
$\sigma_s\to0$ (coefficient $\sim\sigma^{4}/(4\sigma_s^{2})$), reflecting the
absence of contraction; the Wasserstein estimate replaces geometric
forgetting by the expansion factor $e^{LT_\delta}$.

\item[\textnormal{(d)}] Flow matching. $(\sigma_s=0$, velocity
parametrization$)$: no score is used and the velocity error enters directly,
\begin{equation}\label{eq:est-flowmatching}
    W_2\big(\rho_\delta,\hat\rho_\delta\big)
    \ \le\
    e^{LT_\delta}\Big[\,W_2(\rho_1,\hat\rho_1)
    +\int_\delta^T\!\eta_t\,\di t\Big]
    \le
    e^{LT_\delta}\big[\,W_2(\rho_1,\hat\rho_1)+\sqrt{T_\delta}\,\eta\,\big];
\end{equation}
case \textnormal{(c)} is the special case $\eta_t=A\varepsilon_t$ of this
transport estimate.

\item[\textnormal{(e)}] \emph{Stochastic interpolant} $(\sigma_s(t)$
free, velocity learned, score used for the stochastic lift$)$: the drift
error obeys $\|\delta c\|_{L^2_\rho}\le\eta_t+\tfrac12\sigma_s^{2}\varepsilon_t$,
giving
\begin{equation}\label{eq:est-interpolant}
    \DKL\big(\rho_\delta\,\|\,\hat\rho_\delta\big)
    \ \lesssim\
    C_{\mathrm{disc}}(\sigma_s)\,h^{2}
    +\frac{1}{\sigma_s^{2}}\Big(\eta+\tfrac12\sigma_s^{2}\varepsilon\Big)^{2}
    +e^{-\sigma_s^{2}T_\delta/(2C_{\mathrm{LS}})}\,\DKL(\rho_1\,\|\,\hat\rho_1).
\end{equation}
If the interpolant velocity is formed from the score, the schedule ties the
channels, $\eta_t=c(t)\varepsilon_t$ by \eqref{eq:eta-eps}, and the middle
term becomes the pointwise-in-time coefficient
$\big(c(t)+\tfrac12\sigma_s^2\big)^{2}\big/\sigma_s^{2}$ against
$\varepsilon_t^2$. The optimization over $\sigma_s$ is taken up in
Section \ref{sec:optimal-diffusion}.
\end{enumerate}
\end{theorem}

\begin{proof}
Cases (a)--(b) are \eqref{eq:three-term} with the aligned drift error
$(A+A_s)\varepsilon_t$ of case (i) of the dictionary, i.e.\ the coefficient
\eqref{eq:Psi-sigma}, and the forgetting rate $A_s/C_{\mathrm{LS}}$, bounding
$\sigma^{2}\le\bar\sigma^{2}$ in the source; in (a), where
$\sigma_s=\sigma$, the exponent uses $\sigma^{2}\ge\underline\sigma^{2}$,
while in (b) the exponent retains the chosen constant $\sigma_s$. For
(c)--(d), at $\sigma_s=0$ the flow is deterministic, so by
Remark \ref{rem:deterministic-limit} the relative entropy has no
contraction; instead one estimates $W_2$ along the two flows. The standard
transport inequality
$\tfrac{\di}{\di t}W_2(\rho_t,\hat\rho_t)\le\|\hat v-v_{c}\|_{L^2_\rho}
+L\,W_2(\rho_t,\hat\rho_t)$ and Gr\"onwall give the displayed bounds, with
$\eta_t=\|\hat v-v_{c}\|_{L^2_\rho}=\tfrac12\sigma^{2}\varepsilon_t$ in the
score parametrization (c) and $\eta_t$ the direct velocity error in (d); the
second inequality in each line is Cauchy--Schwarz in $t$,
$\int_\delta^T\eta_t\,\di t\le\sqrt{T_\delta}\,\eta$.
For (e), the drift error bound
$\|\delta c\|\le\eta_t+\tfrac12\sigma_s^{2}\varepsilon_t$
enters the Young step of Proposition \ref{prop:adjustable-bound} with
dissipation $A_s=\tfrac12\sigma_s^{2}$, giving the source coefficient
$\tfrac{1}{2A_s}(\eta+A_s\varepsilon)^2
=\tfrac{1}{\sigma_s^{2}}(\eta+\tfrac12\sigma_s^{2}\varepsilon)^{2}$; the
schedule-tied form substitutes $\eta_t=c(t)\varepsilon_t$ pointwise.
\end{proof}

\begin{table}[t]
\centering
\small
\setlength{\tabcolsep}{4.5pt}
\begin{tabular}{lcccl}
\toprule
Framework & $\sigma_s$ & metric & dominant error coefficient & forgetting / growth\\
\midrule
Score SDE / DDPM        & $\sigma$        & $\DKL$ & $\bar\sigma^{2}\,\varepsilon^{2}$
    & $e^{-\underline\sigma^{2}T_\delta/(2C_{\mathrm{LS}})}$\\
General SDE             & $\sigma_s>0$    & $\DKL$ & $\dfrac{(\bar\sigma^{2}+\sigma_s^{2})^{2}}{4\sigma_s^{2}}\,\varepsilon^{2}$
    & $e^{-\sigma_s^{2}T_\delta/(2C_{\mathrm{LS}})}$\\[2ex]
PF-ODE                  & $0$             & $W_2$  & $\tfrac12\bar\sigma^{2}\sqrt{T_\delta}\,\varepsilon$
    & $e^{LT_\delta}$\\
Flow matching           & $0$             & $W_2$  & $\sqrt{T_\delta}\,\eta$
    & $e^{LT_\delta}$\\
Stochastic interpolant  & $\sigma_s$      & $\DKL$ & $\dfrac{1}{\sigma_s^{2}}\big(\eta+\tfrac12\sigma_s^{2}\varepsilon\big)^{2}$
    & $e^{-\sigma_s^{2}T_\delta/(2C_{\mathrm{LS}})}$\\
\bottomrule
\end{tabular}
\caption{Uniform error estimates in the two channels $(\eta,\varepsilon)$
(Theorem \ref{thm:unified-error}). Stochastic samplers ($\sigma_s>0$) are
controlled in relative entropy with geometric forgetting at rate
$\sigma_s^{2}/(2C_{\mathrm{LS}})$; deterministic samplers ($\sigma_s=0$) are
controlled in Wasserstein distance with the flow-expansion factor $e^{LT_\delta}$.
The score-error coefficient $(\sigma^{2}+\sigma_s^{2})^{2}/(4\sigma_s^{2})$ is
minimized at $\sigma_s=\sigma$, where it equals $\sigma^{2}$; the interpolant's
velocity channel shifts the optimum to $\sigma_{s,\star}^{2}=2\eta/\varepsilon$
(Section \ref{sec:optimal-diffusion}).}
\label{tab:per-framework}
\end{table}

\noindent
Three features are visible in Table \ref{tab:per-framework}. First, the
discretization constant $C_{\mathrm{disc}}$ is smallest for the deterministic
samplers, which admit higher-order integrators; this is the regime in which
ODE sampling is preferable when the score (or velocity) is accurate. Second,
only the stochastic samplers forget their initialization, at a rate growing
with $\sigma_s^{2}$, so a nonzero noise level is required whenever
$\DKL(\rho_T\,\|\,\rho_\infty)$ is not already negligible. Third, the
score/velocity coefficient has an interior minimum in $\sigma_s$, at
$\sigma_s=\sigma$ for the score-parametrized samplers and at
$\sigma_{s,\star}^{2}=2\eta/\varepsilon$ for the velocity-parametrized
interpolant; in both cases the deterministic limit ($\sigma_s\to0$)
sacrifices robustness to learning error in exchange for the smaller
discretization constant, recovering the empirical ODE/SDE trade-off from the
single inequality \eqref{eq:three-term}. The next subsection makes the
interior minimum precise.

\subsection{The optimal diffusion}
\label{sec:optimal-diffusion}

The score/velocity term is the only term of \eqref{eq:three-term} with an
interior minimum in the sampler diffusion: the discretization term favours
$A_s=0$ and the initialization term favours $A_s$ large. We minimize its
coefficient, first for the general SDE sampler with independent error
channels, then for the stochastic interpolant, whose schedule fixes the
ratio $\eta_t/\varepsilon_t$.

\begin{corollary}
\label{cor:optimal-diffusion}
Fix $t$ and minimize the coefficient of \eqref{eq:three-term},
\begin{equation*}
    \Psi(A_s):=\frac{1}{2A_s}\big(\eta_t+A_s\varepsilon_t\big)^2
    =\frac{\eta_t^2}{2A_s}+\eta_t\varepsilon_t+\frac{A_s}{2}\,\varepsilon_t^2,
\end{equation*}
over $A_s>0$. Then
\begin{equation}\label{eq:optimal-epsilon}
    A_s^{*}(t)=\frac{\eta_t}{\varepsilon_t},
    \qquad
    \Psi(A_s^{*})=2\,\eta_t\,\varepsilon_t :
\end{equation}
the optimal sampler diffusion is the ratio of the current-velocity error to
the score error, the optimal-noise result of the stochastic interpolant
framework \cite{albergo2023stochastic}. For the score-parametrized sampler,
where $\eta_t=A\varepsilon_t$ by \eqref{eq:eta-eps-score}, this specializes to
\begin{equation}\label{eq:optimal-score-based}
    A_s^{*}=A=\tfrac12\sigma\sigma^{T},
    \qquad \sigma_s^*=\sigma :
\end{equation}
the optimal reverse noise equals the forward noise, i.e.\ the standard
reverse SDE, and Theorem \ref{thm:unified-error}\textnormal{(a)} sits exactly
at the minimum of the coefficient \eqref{eq:Psi-sigma}. The deterministic
limit $A_s\to0$ makes $\Psi(A_s)\sim\eta_t^2/(2A_s)\to\infty$, and the
over-noised limit $A_s\to\infty$ makes
$\Psi(A_s)\sim\tfrac{A_s}{2}\varepsilon_t^2\to\infty$.
\end{corollary}

\begin{proof}
The function $\Psi$ is convex on $(0,\infty)$, with derivative
$\Psi'(A_s)=-\tfrac{\eta_t^2}{2A_s^2}+\tfrac12\varepsilon_t^2$
vanishing at $A_s=\eta_t/\varepsilon_t$; substituting this value gives
\eqref{eq:optimal-epsilon}. For the score-parametrized sampler the drift
error is the aligned $(A+A_s)\varepsilon_t$ of case (i), so
$\eta_t=A\varepsilon_t$ and the minimizer becomes $A_s^{*}=A$, which is
\eqref{eq:optimal-score-based}.
\end{proof}

\begin{corollary}
\label{cor:interpolant-optimal}
Let the interpolant sampler form its velocity from the score, so that
$\eta_t=c(t)\varepsilon_t$ with the schedule coefficient $c(t)$ of
Lemma \ref{lem:velocity-score} and the drift error is
$(c(t)+A_s)\,\delta s$. The score-error contribution to
\eqref{eq:three-term} is then
\begin{equation}\label{eq:interp-coeff}
    \int_\delta^T\frac{\big(c(t)+A_s\big)^2}{2A_s}
    \,\varepsilon_t^2\,\di t ,
\end{equation}
whose pointwise coefficient is minimized, for each $t$, at
\begin{equation}\label{eq:eps-star}
    A_s^{*}(t)=c(t),\qquad
    \sigma_{s,*}^2(t)=2\,c(t)=2\,\beta_t^2\,\frac{\di}{\di t}\log\frac{\alpha_t}{\beta_t},
\end{equation}
with minimal value $2\,c(t)$. Consequently, with the schedule-matched noise
\eqref{eq:eps-star} and $\underline c:=\inf_{[\delta,T]}c$,
$\bar c:=\sup_{[\delta,T]}c$,
\begin{equation}\label{eq:interp-sharp}
    \DKL\big(\rho_\delta\,\|\,\hat\rho_\delta\big)
    \ \lesssim\
    C_{\mathrm{disc}}\,h^{2}
    +2\,\bar c\,\varepsilon^{2}
    +e^{-\underline c\,T_\delta/C_{\mathrm{LS}}}\,\DKL(\rho_1\,\|\,\hat\rho_1).
\end{equation}
By \eqref{eq:c-equals-A}, for an admissible schedule $c(t)=A(t)$ and
\eqref{eq:eps-star} reduces to $\sigma_{s,*}=\sigma$: the interpolant
optimum contains the general-SDE optimum \eqref{eq:optimal-score-based} as
the special case of a disguised diffusion. If instead the velocity and the
score are learned separately, the channels are independent and the optimum
reverts to $A_s^{*}=\eta_t/\varepsilon_t$ of \eqref{eq:optimal-epsilon}:
separate learning frees the optimal noise from the forward diffusion and is
advantageous when the two errors differ in size.
\end{corollary}

\begin{proof}
The drift error $(c(t)+A_s)\delta s$ enters the Young step of
Proposition \ref{prop:adjustable-bound} with dissipation $A_s$, giving the
source coefficient $\Psi(A_s)=\tfrac{(c+A_s)^2}{2A_s}$, which is
\eqref{eq:interp-coeff}. Since $c>0$ in the generation direction,
$\Psi'(A_s)=\tfrac{(c+A_s)(A_s-c)}{2A_s^2}=0$ at $A_s=c$, with
$\Psi(c)=2c$. Bounding the coefficient by $2\bar c$ and the
(time-dependent) forgetting rate $c(t)/C_{\mathrm{LS}}$ from below by
$\underline c/C_{\mathrm{LS}}$ yields \eqref{eq:interp-sharp}.
\end{proof}

The three terms of \eqref{eq:three-term} thus pull in different directions
and reproduce the ODE/SDE trade-off from a single inequality. The
discretization term $C_{\mathrm{disc}}(A_s)h^{2}$ favours $A_s=0$; the
initialization term favours $A_s>0$, with forgetting at rate
$\propto A_s$; and the score/velocity term is minimized at the interior
schedule \eqref{eq:optimal-epsilon}--\eqref{eq:eps-star}. When the learned
fields are accurate the discretization term dominates and the optimal
sampler is close to deterministic; otherwise a stochastic sampler near
$A_s^{*}$ is preferable. The empirically observed optimality of an
intermediate level of injected noise \cite{karras2022elucidating} is the
balance of these three terms, and $A_s^{*}(t)$ is its quantitative form.

\section{Numerical Experiments}\label{sec:Numerical}
In this section, we present six numerical experiments to investigate the convergence behavior of diffusion model samplers. We analyze the weak convergence by measuring the KL divergence between the generated distribution and the target distribution as a function of the discretization step size, $\Delta t$. The experiments are designed to test the performance of both the analytical Euler-Maruyama (EM) scheme and a learned score-based generative model (SGM) on different target distributions. As all of the Euler-Maruyama data, the training and output data of the score-based generative model are i.i.d samples from each distribution. For non-Gaussian targets we estimate the KL divergence with the non-parametric estimator of P\'erez-Cruz \cite{perezcruz2008kl}, defined as
\begin{equation*}
    \hat{D}(P||Q) = \frac{1}{n}\sum_{i=1}^n\log \frac{\delta P_c(x_i)}{\delta Q_c(x_i)},
\end{equation*}

where $P_c(x), Q_c(x)$ are the continuous piece-wise linear extensions to their empirical cumulative distribution function(ecdf) and $\delta P_c(x_i) = P_c(x_i) -P_c(x_i -\varepsilon)$ for any $\varepsilon \le \min_i(x_i- x_{i-1})$. 

And it is shown that
\begin{equation*}
    \hat{D}(P||Q) -1 \xrightarrow{a.s} D(P||Q) \quad as\quad n \to \infty.
\end{equation*}

\subsection{Experiment 1: Euler-Maruyama Scheme for OU process with initial distribution taking Gaussian form}
First, we consider a standard OU process where the initial distribution is Gaussian:
$$\di X_t = -\frac{1}{2}X_t\di t + \di W_t,\quad X_0 \sim \mathcal{N}(\mu, \sigma^2).$$
The distribution of $X_t$ at any time $t$ is also Gaussian, given by $X_t \sim \mathcal{N}(\mu e^{-t/2}, \sigma^2 e^{-t} + 1 - e^{-t})$. The reverse-time SDE, which transforms the distribution $\rho_T$ back to $\rho_0$, is given by:
\begin{equation*}
\di \hat{X}_t = \left[ \frac{1}{2}\hat{X}_t - \frac{\hat{X}_t - \mu e^{-t/2}}{\sigma^2 e^{-t} + 1-e^{-t}} \right] \di t + \di \bar{W}_t,
\end{equation*}
where $t$ runs backwards from $T$ to $0$. We discretize this SDE using the Euler-Maruyama scheme. The accuracy of the generated samples, which should follow $\mathcal{N}(\mu, \sigma^2)$, is measured using the analytical KL divergence formula for two normal distributions $\rho_1 = \mathcal{N}(\mu_1, \sigma_1^2)$ and $\rho_2 = \mathcal{N}(\mu_2, \sigma_2^2)$:
\begin{equation*}
KL(\rho_1||\rho_2) = \log\frac{\sigma_2}{\sigma_1}+ \frac{\sigma_1^2 +(\mu_1 - \mu_2)^2}{2\sigma_2^2} - \frac{1}{2}.
\end{equation*}

For this experiment, we set the target distribution as $\rho_0 \sim \mathcal{N}(0.5, 0.6^2)$ and the diffusion time $T=2$. We generate $10^6$ samples for various step sizes $\Delta t$. We test two initialization strategies for the backward process at time $T$:
\begin{enumerate}
    \item \textbf{True Initial Condition:} Sampling from the exact distribution $$\rho_T = \mathcal{N}(\mu e^{-T/2}, \sigma^2 e^{-T} + 1 - e^{-T})$$.
    \item \textbf{Gaussian Initial Condition:} Sampling from the stationary distribution $\mathcal{N}(0, 1)$.
\end{enumerate}

The results are shown in Figure \ref{fig:em_gaussian}.

\begin{figure}[ht!]
    \centering
    \maybeincludegraphics[width=0.48\textwidth]{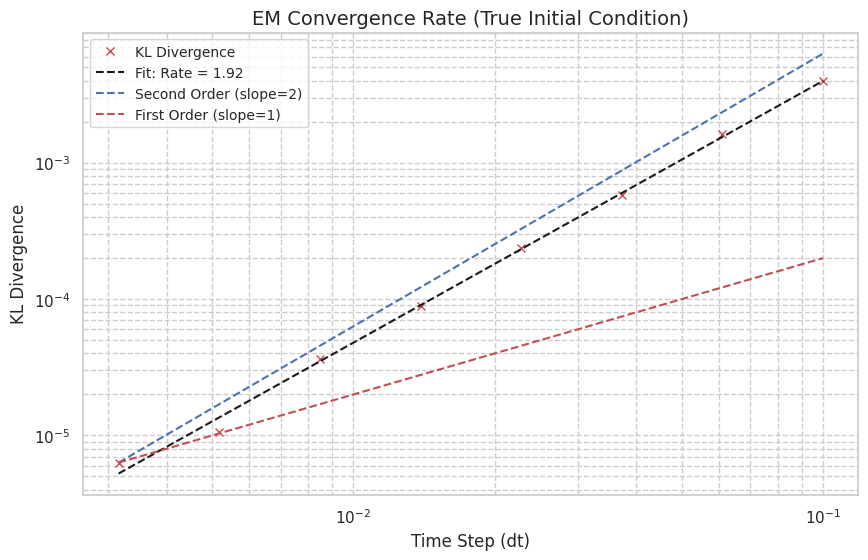}
    \maybeincludegraphics[width=0.48\textwidth]{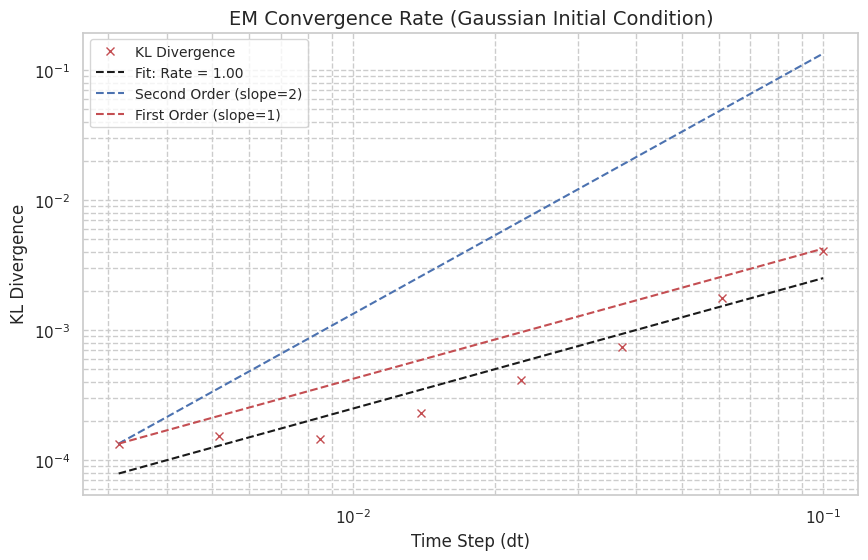}
    \caption{Convergence of the Euler-Maruyama scheme for a Gaussian target. Left: Starting from the true distribution at $t=T$. Right: Starting from the stationary distribution $\mathcal{N}(0,1)$. The black dashed line shows the fitted convergence rate, while the red and blue dashed lines show reference slopes for first-order and second-order convergence, respectively.}
    \label{fig:em_gaussian}
\end{figure}

The log-log plots in Figure \ref{fig:em_gaussian} show that the KL divergence decreases as $\Delta t$ becomes smaller. The fitted slope is approximately 2.0 in both cases, indicating $\DKL=\mathcal{O}(\Delta t^2)$. The error when starting from the stationary distribution is slightly higher, as it includes an additional initialization error term that decays exponentially with $T$.

\subsection{Experiment 2: Euler-Maruyama Scheme for VP process with initial distribution taking Gaussian form}

We consider the following VP process
$$\di X_t = -\frac{t}{2T}X_t\di t + \sqrt{\frac{t}{T}}\di W_t,\quad X_0 \sim \rho_0.$$

If the initial distribution is Gaussian, $X_0 \sim \mathcal{N}(\mu, \sigma^2\mathcal{I})$, then the distribution of $X_t$ at any time $t$ is also Gaussian:
$$X_t \sim \mathcal{N}(\mu e^{-t^2/4T}, \Sigma e^{-t} + (1 - e^{-t^2/2T})).$$

The score function, $\nabla_x \log \rho(x,t)$, is then:
$$\nabla_x \log \rho(x,t) = -\frac{x - \mu e^{-t^2/4T}}{\Sigma e^{-t^2/2T} + (1-e^{-t^2/2T})}.$$

The corresponding reverse-time SDE is given by:
$$\di \hat{X}_t = \frac{t}{T}\left[ \frac{1}{2}\hat{X}_t + \nabla_x \log \rho(\hat{X}_t, t) \right] \di t + \sqrt{\frac{t}{T}}\di \bar{W}_t,$$
$$\di \hat{X}_t = \left[ \frac{t}{2T}\hat{X}_t - \frac{t}{T}\frac{\hat{X}_t - \mu e^{-t^2/4T}}{\Sigma e^{-t^2/2T} + (1-e^{-t^2/2T})} \right] \di t + \sqrt{\frac{t}{T}}\di \bar{W}_t, \quad \hat{X}_T \sim \rho_T.$$

For this experiment, we set the target distribution as $\rho_0 \sim \mathcal{N}(0.5, 0.7^2)$ and the diffusion time $T=5$. We generate $10^6$ samples for various step sizes $\Delta t$.

The left panel in Figure \ref{fig:vp_gaussian} shows the $\mathcal{O}(\Delta t^2)$ KL scaling of the Euler-Maruyama scheme for the VP process, while the right panel shows the $\mathcal{O}(e^{-T})$ convergence for the initialization error.
\begin{figure}[ht!]
    \centering
    \maybeincludegraphics[width=0.48\textwidth]{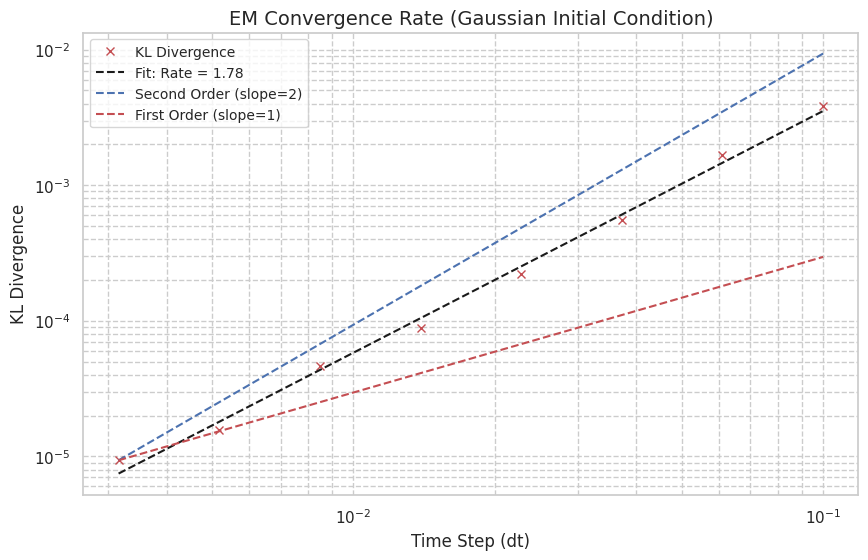}
        \maybeincludegraphics[width=0.48\textwidth]{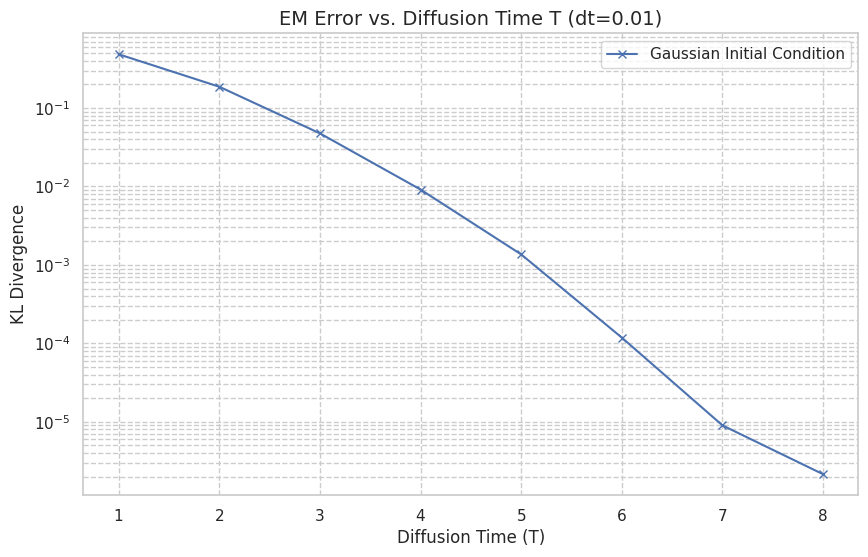}
    \caption{Left: KL divergence of the EM sampler for the Gaussian target for the VP process in various time-steps; Right: The KL divergence of the EM sampler for the Gaussian target for the VP process in various T}
    \label{fig:vp_gaussian}
\end{figure}

Figure \ref{fig:vp_d} shows the linear complexity $\mathcal{O}(n)$ for the dimension.
\begin{figure}[ht!]
    \centering
    \maybeincludegraphics[width=0.48\textwidth]{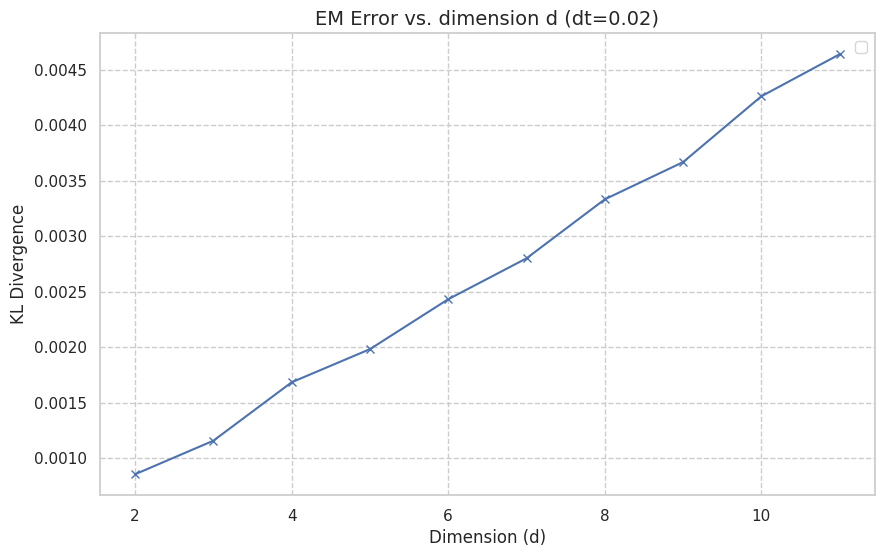}
    \caption{ The KL divergence of the EM sampler for the Gaussian target for the VP process in various dimensions $n$.}
    \label{fig:vp_d}
\end{figure}

\subsection{Experiment 3: Score-Based Generative Model with a Gaussian Target}

In the third experiment, we replace the analytical score with a neural network approximation $s_\theta(\mathbf{x}, t)$, which is trained to predict the noise $\epsilon$ added during the forward process. We use the same OU process as in the previous section:
\begin{equation*}
    d X_t = -\frac{1}{2}X_t d t + d W_t,\quad X_0 \sim \mathcal{N}(\mu,\sigma^2).
\end{equation*}
The backward process is approximated using the Euler-Maruyama scheme applied to the learned SDE, with $x_0 \sim \mathcal{N}(0,1)$:
  \begin{equation*}
      x_{k-1} = x_k + \left(\frac{1}{2} x_k + s_\theta(x_k, t_k)\right)\Delta t + \sqrt{\Delta t} \xi_k,
 \end{equation*}
where we use the relation $s_\theta(x_k, t_k) \approx -\epsilon_\theta(x_k, t_k) / \sqrt{1 - e^{-t_k}}$ and $\Delta t = t_k - t_{k-1}$. The model $\epsilon_\theta$ is a neural network trained on the same target distribution, $\mathcal{N}(0.5, 0.6^2)$, with $T=5$.

\begin{figure}[ht!]
    \centering
    \maybeincludegraphics[width=0.48\textwidth]{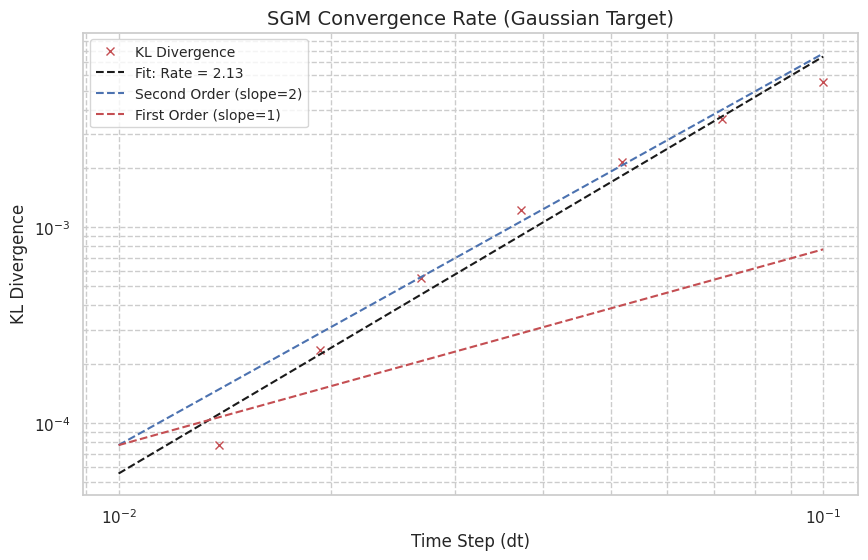}
    \maybeincludegraphics[width=0.48\textwidth]{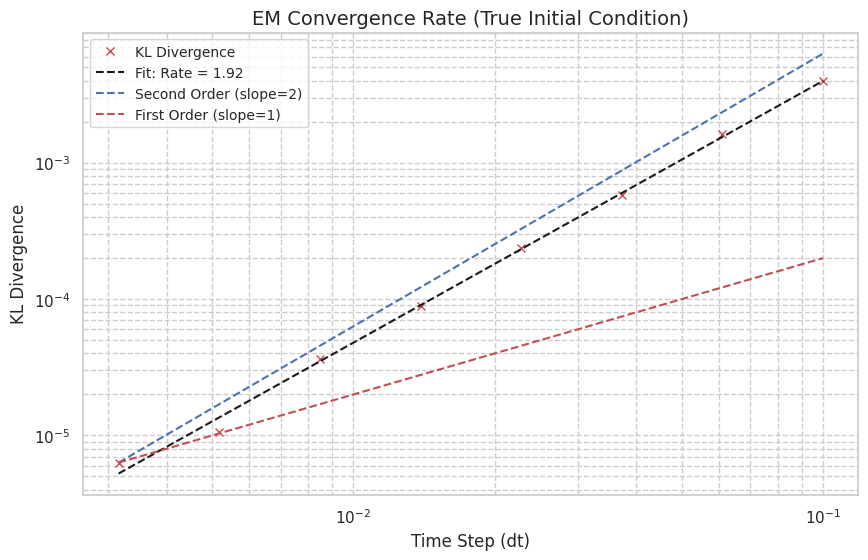}
    \caption{Left: Generated distribution from the SGM (red) versus the true Gaussian target (blue). Right: Convergence of the SGM sampler for the Gaussian target, with fitted KL scaling close to $\mathcal{O}(\Delta t^2)$.}
    \label{fig:sgm_gaussian}
\end{figure}

The left panel of Figure \ref{fig:sgm_gaussian} shows that the trained model accurately reproduces the target Gaussian distribution. The right panel shows the convergence behavior of the SGM sampler. The fitted slope is approximately 1.92, indicating $\DKL\approx\mathcal{O}(\Delta t^{1.92})$, close to the $\mathcal{O}(\Delta t^2)$ scaling predicted for the discretization contribution. Experiment 5 below is designed to separate the two error sources by varying the score error at fixed $h$.

\subsection{Experiment 4: Score-Based Generative Model with Beta Distributions}

To explore performance on non-Gaussian targets, we now consider initial distributions from the Beta family, $X_0 \sim B(\alpha, \beta)$, which have compact support on $[0, 1]$. We employ a Variance Preserving (VP) SDE with a linear noise schedule, a standard choice in modern diffusion models. The reverse SDE is given by:
\begin{equation*}
    \di \hat{X}_t = \left[-\frac{1}{2} \beta(t) \hat{X}_t - \beta(t) s_\theta(\hat{X}_t, t) \right] \di t + \sqrt{\beta(t)} \di \bar{W}_t.
\end{equation*}
We test two cases: a smooth, bell-shaped distribution and a singular, U-shaped distribution. The KL divergence is estimated non-parametrically using the Perez-Cruz estimator.

We first train a score model targeting a $B(2,2)$ distribution. This distribution is symmetric, and its density smoothly goes to zero at the boundaries, making it a relatively well-behaved target.

\begin{figure}[ht!]
    \centering
    \maybeincludegraphics[width=0.48\textwidth]{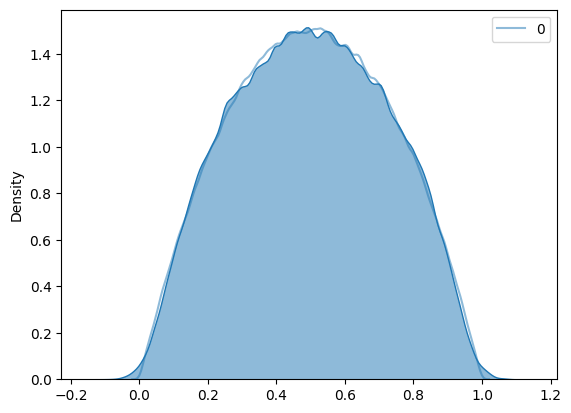}
    \maybeincludegraphics[width=0.48\textwidth]{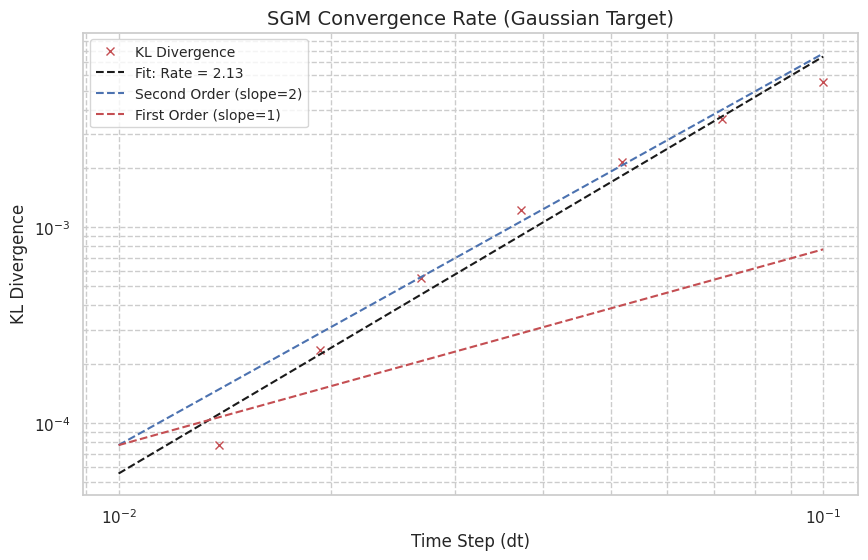}
    \caption{Left: Generated distribution from the SGM (red) versus the true Beta(2,2) target (blue). Right: Convergence of the SGM sampler for the Beta(2,2) target, showing fitted KL scaling close to $\mathcal{O}(\Delta t^{1.8})$.}
    \label{fig:sgm_beta22}
\end{figure}

As shown in Figure \ref{fig:sgm_beta22} (left), the model successfully learns to generate samples that closely match the target $B(2,2)$ distribution. The convergence plot (right) shows a fitted slope of $1.8$, indicating $\DKL\approx\mathcal{O}(\Delta t^{1.8})$. This result demonstrates that the score-based framework and the VP-SDE sampler are effective for learning and generating samples from smooth, compactly supported distributions.

Finally, we consider a more challenging target, the $B(0.5, 0.5)$ distribution. This U-shaped distribution has infinite density at its boundaries (0 and 1), which makes its score function singular and difficult for a neural network to approximate.

\begin{figure}[ht!]
    \centering
    \maybeincludegraphics[width=0.48\textwidth]{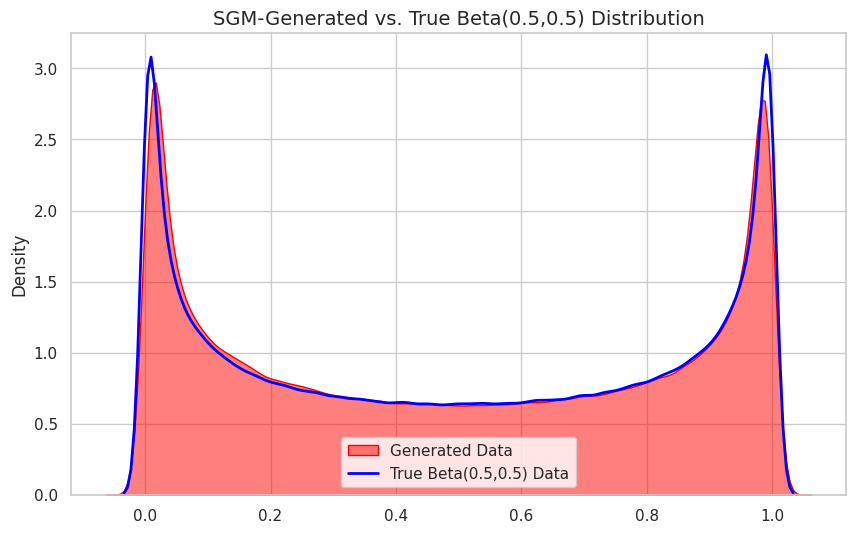}
    \maybeincludegraphics[width=0.48\textwidth]{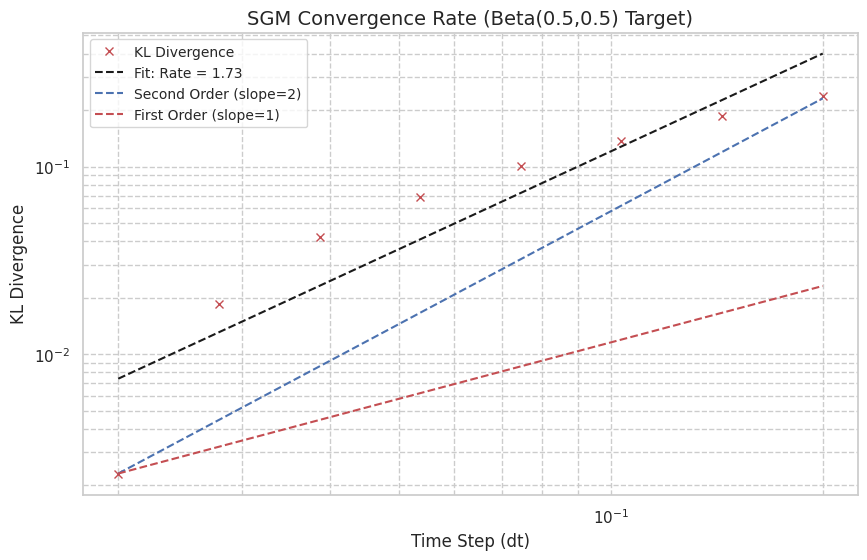}
    \caption{Left: Generated distribution from the SGM (red) versus the true Beta(0.5,0.5) target (blue). Right: Convergence of the SGM sampler for the Beta(0.5,0.5) target, showing fitted KL scaling close to $\mathcal{O}(\Delta t^{1.86})$.}
    \label{fig:sgm_beta0505}
\end{figure}

Figure \ref{fig:sgm_beta0505} (left) shows that despite the singularities, the SGM can capture the characteristic U-shape of the target distribution, although with less precision at the boundaries compared to the smooth case. The convergence plot (right) shows a fitted slope of $1.86$, indicating $\DKL\approx\mathcal{O}(\Delta t^{1.86})$. The absolute KL divergence is higher than in the Beta(2,2) case across all step sizes, which reflects the increased difficulty of approximating the singular score function. This experiment highlights the robustness of the score-based method even for challenging, singular target distributions. The theoretical reading is the following: the $B(0.5,0.5)$ target violates the global Fisher-information assumption (its density $\propto x^{-1/2}(1-x)^{-1/2}$ has infinite Fisher information at the boundary), and the finite, decreasing KL observed here is exactly the behavior predicted by the early-stopped Corollary \ref{cor:early-stopped}, with the implicit early stopping supplied by the schedule floor of the VP discretization.

\subsection{Experiment 5: KL divergence versus the score error}
\label{sec:exp-score-error}

The bound of Theorem \ref{thm:total-error} is quadratic in the score error:
at a fixed step size $h$, the KL divergence should behave as
$C_1h^2+C_2\varepsilon^2$, so that above the discretization floor it grows
linearly in $\varepsilon^2$. Isolating this dependence with trained networks
is delicate, because the training error is neither exactly known nor freely
adjustable. We therefore test the prediction in a setting where the score
error is prescribed analytically: we take the OU process and Gaussian target
of Experiment 1, whose score
$\nabla\log\rho(x,t)=-(x-m_t)/v_t$ is known in closed form with
$m_t=\mu e^{-t/2}$ and $v_t=\sigma^2e^{-t}+1-e^{-t}$, and run the sampler
with a \emph{deliberately perturbed} score. Two families of perturbation are
used,
\begin{equation}\label{eq:score-perturbations}
    \hat s=\nabla\log\rho+\kappa
    \quad\text{and}\quad
    \hat s=(1+\lambda)\nabla\log\rho ,
\end{equation}
the first shifting the generated mean and the second its variance, with
score errors available in closed form,
\begin{equation}\label{eq:eps-perturbations}
    \varepsilon^2_{\kappa}=\int_0^T\!\!\|\kappa\|^2_{L^2_{\rho_t}}\di t=\kappa^2T,
    \qquad
    \varepsilon^2_{\lambda}=\lambda^2\!\int_0^T\! I(\rho_t)\,\di t
    =\lambda^2\!\int_0^T\!\frac{\di t}{v_t},
\end{equation}
since $\|\nabla\log\rho_t\|^2_{L^2_{\rho_t}}$ is the Fisher information
$1/v_t$ of $\mathcal N(m_t,v_t)$.

The experiment requires neither training nor sampling. Both perturbed drifts
remain affine in $x$, so each Euler--Maruyama step is an affine map plus
Gaussian noise and the law of the sampler stays Gaussian throughout the run;
its mean and variance obey the two-term recursion
$m\mapsto am+b$, $v\mapsto a^2v+h$ with
$a=1+h\big(\tfrac12-(1+\lambda)/v_t\big)$ and
$b=h\big((1+\lambda)m_t/v_t+\kappa\big)$, and the KL divergence to the target
is then evaluated exactly by the Gaussian formula. The reported values are
therefore free of Monte Carlo error, and the entire experiment runs in
seconds. We initialize from the exact terminal law $\rho_T$, so that the
initialization term of Theorem \ref{thm:total-error} vanishes and the bound
reduces to the two terms under test.

\begin{figure}[ht!]
    \centering
    \maybeincludegraphics[width=0.48\textwidth]{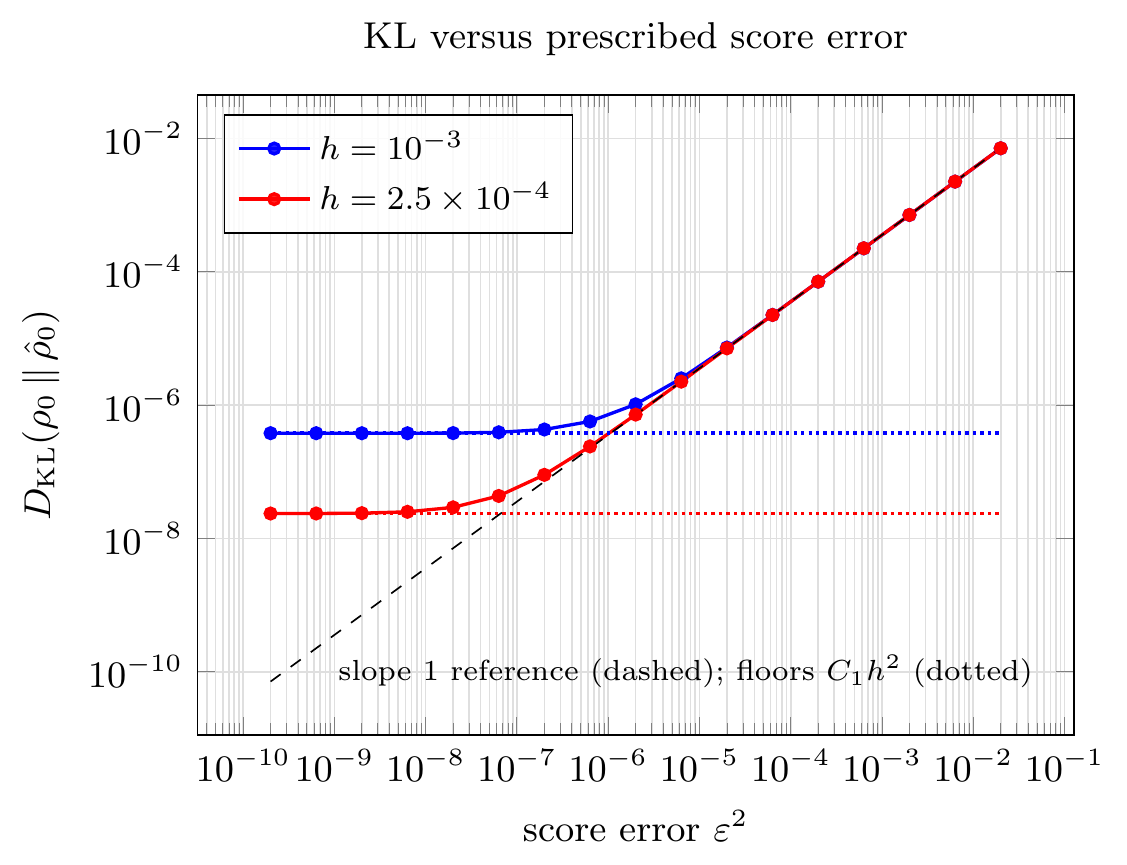}
    \maybeincludegraphics[width=0.48\textwidth]{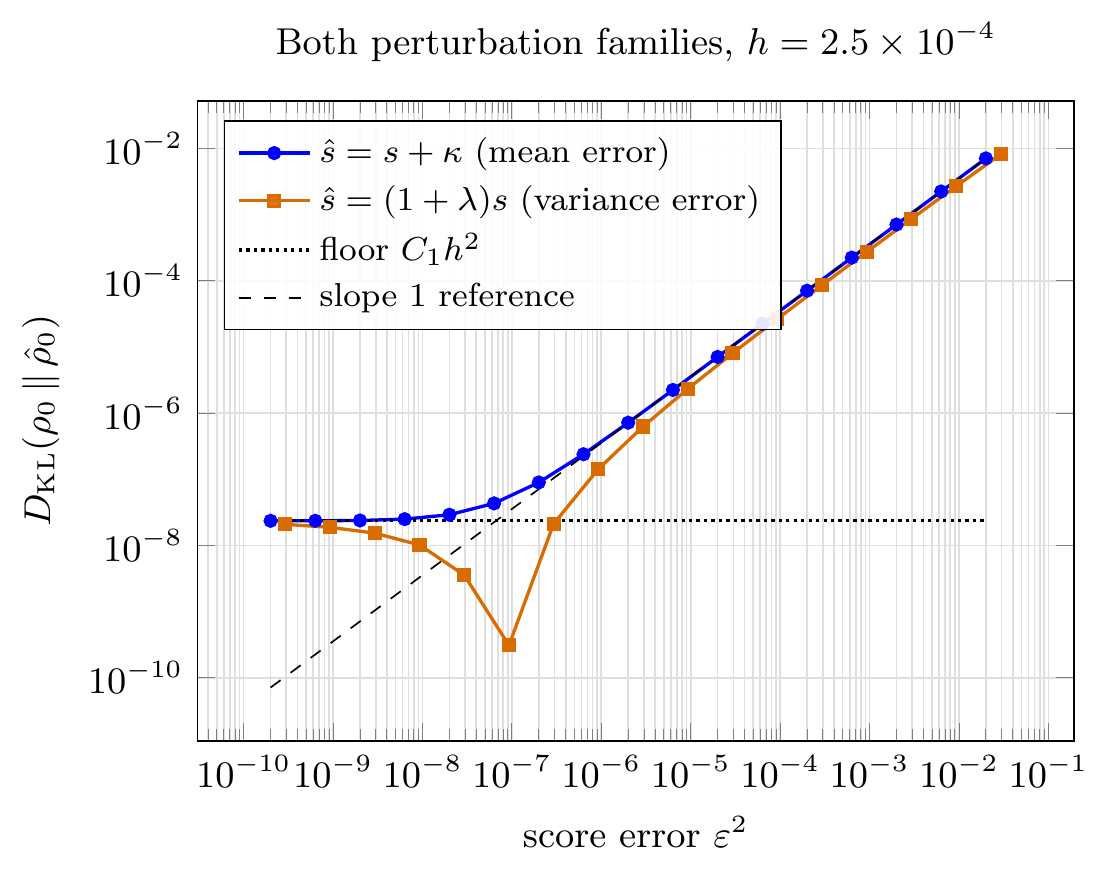}
    \caption{KL divergence versus the prescribed score error $\varepsilon^2$
    for the Gaussian target with a perturbed analytical score (log-log).
    Left: $\DKL$ against $\varepsilon^2_\kappa$ for two step sizes; the
    dotted lines mark the discretization floors $C_1h^2$, and above them the
    data follow the slope-$1$ prediction of Theorem \ref{thm:total-error}.
    Right: both perturbation families of \eqref{eq:score-perturbations} at
    $h=2.5\times10^{-4}$; they coincide above the floor, showing that the
    error enters only through $\varepsilon^2$. The dip of the variance
    family near $\varepsilon^2\approx10^{-7}$ is the partial cancellation
    discussed in the text.}
    \label{fig:kl_vs_eps}
\end{figure}

The measurements confirm the predicted decomposition quantitatively. For
$h=10^{-3}$ the floor is $C_1h^2=3.79\times10^{-7}$ and for
$h=2.5\times10^{-4}$ it is $2.37\times10^{-8}$: quartering the step lowers
the floor by a factor $16.0$, the exact $h^2$ ratio. Above the floor the
fitted slopes against $\varepsilon^2$ are $0.998$ and $1.001$ respectively,
and the fitted constant $C_2=\DKL/\varepsilon^2$ is $0.356$ and $0.355$,
i.e.\ stable across three decades of $\varepsilon^2$ and independent of the
step size, as the additive form $C_1h^2+C_2\varepsilon^2$ requires. The
multiplicative family gives the slope $1.003$ and, once above the floor,
falls on the same line as the constant family: at equal $\varepsilon^2$ a
perturbation of the mean and a perturbation of the variance produce the same
KL error, so the bound depends on the score error only through its weighted
$L^2$ norm and not through its shape.

One deviation is visible and worth recording. For the variance perturbation
at small $\lambda$ the KL divergence dips \emph{below} the discretization
floor, reaching $3.1\times10^{-10}$ near $\varepsilon^2\approx10^{-7}$,
almost two orders of magnitude under the unperturbed value. The reason is
that the Euler--Maruyama scheme carries a systematic variance bias, and a
multiplicative score error of the appropriate sign cancels part of it; the
two error sources are then not independent, and the sum $C_1h^2+C_2\varepsilon^2$
overestimates the true error. This is consistent with
Theorem \ref{thm:total-error}, which is an upper bound, and it locates
precisely where that bound ceases to be tight: the additive form is sharp
when the two error sources are unrelated, and conservative when a learned
field happens to compensate the discretization bias.

\subsection{Experiment 6: the entropy-production profile}
\label{sec:exp-ep-profile}

The error analysis is organized around the entropy production rate
$e_p(t)$, and two of its predictions are directly measurable in the Gaussian
setting of Experiment 1. There, both the forward marginal $\rho(\cdot,t)$
and the sampler marginal $\rho'(\cdot,T-t)$ are Gaussian with recursively
computable means and variances (each Euler--Maruyama step is an affine map
plus Gaussian noise), so
$D(t)=\DKL(\rho(\cdot,t)\,\|\,\rho'(\cdot,T-t))$ is available in closed form
along the trajectory without sampling error; this exactly solvable Gaussian
setting is the one analyzed, in Wasserstein distance, by
\cite{pierret2025diffusion}. First, the endpoint identity
\eqref{eq:KL-divergence_equality}: the numerically differentiated profile
$e_p(t)=-D'(t)$ must integrate to $D(0)-D(T)$. Second, the localization of
the entropy production: the discretization contribution to $e_p$ is
predicted to concentrate near the data end, where the reverse drift and its
derivatives are largest (Remark \ref{rem:reverse-drift-singularity}); this
concentration is the quantitative motivation for non-uniform step schedules
that concentrate the discretization budget at the data end.

\begin{figure}[ht!]
    \centering
    \maybeincludegraphics[width=0.48\textwidth]{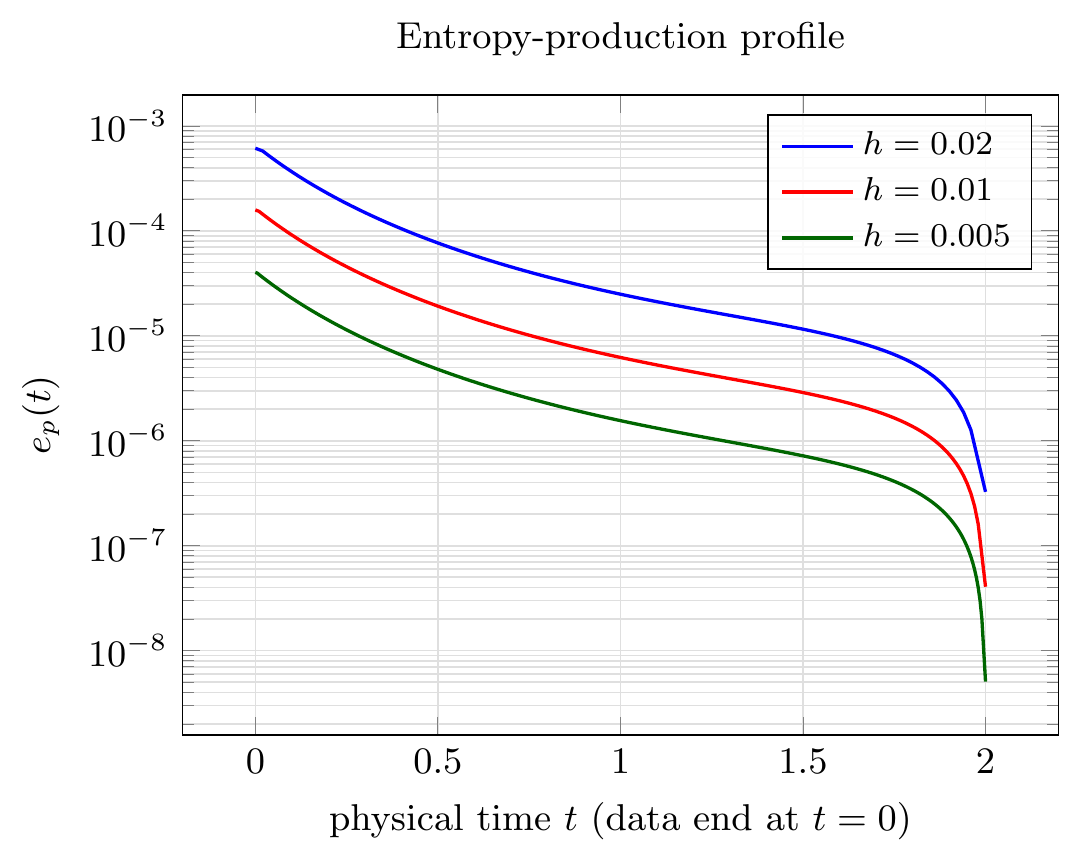}
    \maybeincludegraphics[width=0.48\textwidth]{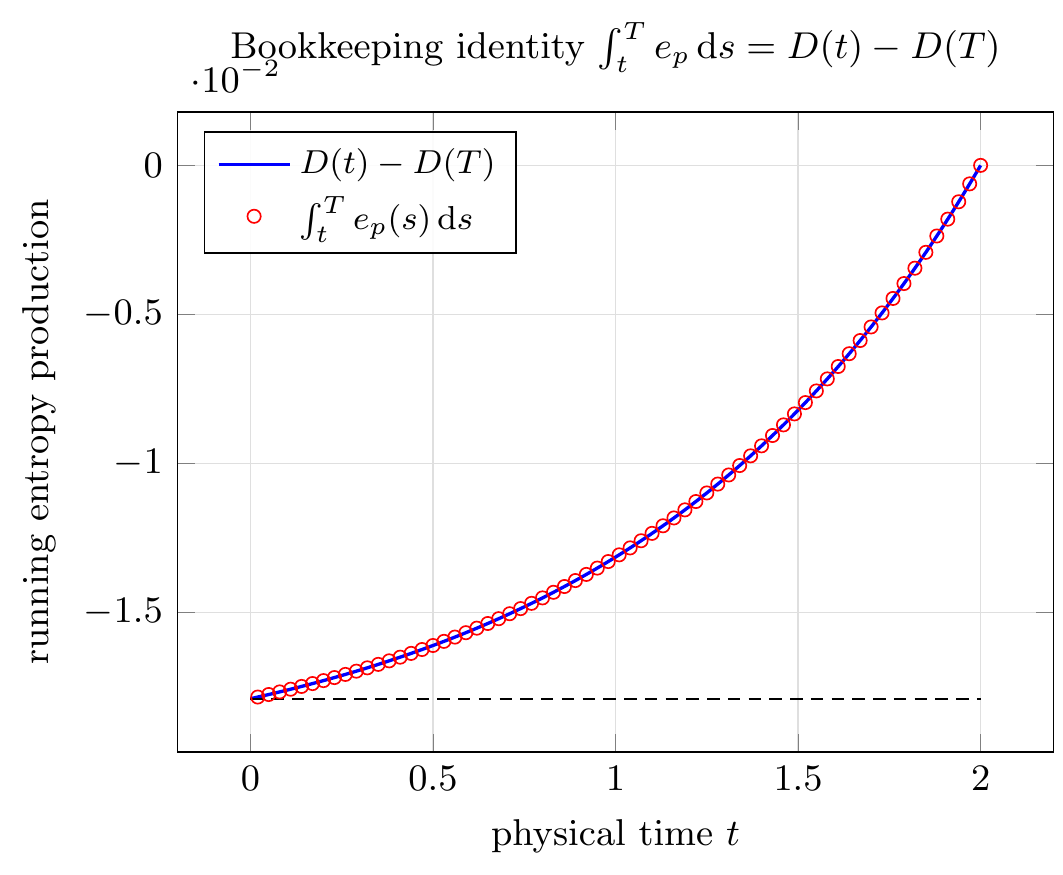}
    \caption{Left: entropy-production profile $e_p(t)$ of the EM sampler for
    the Gaussian OU target at three step sizes, computed by closed-form
    evaluation of $D(t)$ along the trajectory and numerical differentiation,
    with the sampler started from the exact terminal law so that $e_p$ is
    purely the discretization contribution (logarithmic ordinate; the data
    end is $t=0$). Right: the running integral $\int_t^T e_p(s)\,\di s$
    against $D(t)-D(T)$ for a sampler started from $\rho_\infty$, verifying
    the identity \eqref{eq:KL-divergence_equality} at every $t$;
    the endpoint value is $D(0)-D(T)=-1.789\times10^{-2}$.}
    \label{fig:ep_profile}
\end{figure}

Both predictions are borne out. The profile is strongly localized at the
data end: with $h=5\times10^{-3}$ the rate falls from
$e_p(0)=4.03\times10^{-5}$ to $e_p(T)=5.07\times10^{-9}$, nearly four orders
of magnitude across the interval, and the ordering is preserved at every
step size. The peak value scales as $h^2$, $6.13\times10^{-4}$,
$1.58\times10^{-4}$ and $4.03\times10^{-5}$ for $h=2\times10^{-2}$,
$10^{-2}$ and $5\times10^{-3}$, ratios of $3.87$ and $3.93$. So refining
the step reduces the rate uniformly without moving the concentration. Since
the terminal error is the time integral of this profile, a uniform step
spends most of its budget where $e_p$ is negligible; this is the
quantitative form of the argument for non-uniform schedules refined at the
data end.

The identity is verified not merely at the endpoints but along
the whole trajectory: the running integral $\int_t^Te_p(s)\,\di s$ and the
difference $D(t)-D(T)$ agree to a maximum deviation of $1.3\times10^{-7}$ on
values of order $1.8\times10^{-2}$, a relative error of $7\times10^{-6}$
attributable to the numerical differentiation of $D$. In this run the
sampler is initialized from $\rho_\infty$ rather than $\rho_T$, so
$D(T)=1.89\times10^{-2}$ is the initialization mismatch and
$D(0)=1.02\times10^{-3}$ the terminal error; the integrated entropy
production is negative because the sampler forgets more initialization error
than the discretization injects, which is the contraction mechanism of
Proposition \ref{prop:lsi-decay} made visible.

\section{Conclusions}\label{sec:conclusion}

We have developed a unified error analysis of generative diffusion models
organized around a single object: the entropy production rate of a
forward/reverse pair of diffusion processes. The velocity-form identity of
Theorem \ref{thm:mother-identity} holds for state-dependent and unequal
diffusion matrices and expresses the terminal Kullback--Leibler divergence as
the time integral of a closed-form rate, so that every error
source of initialization, score approximation, and time
discretization, is tracked through one unifrom framwork. For the
score-based sampler discretized by the Euler--Maruyama scheme, this yields
the three-part bound of Theorem \ref{thm:total-error},
$\DKL\le C_1(T)h^2+C_2(T)\varepsilon^2+e^{-2T/C}\DKL(\mu_0\|\rho_\infty)$,
whose $\mathcal O(h^2)$ discretization rate---obtained by accumulating the
entropy production at the level of the marginals rather than in path
space is confirmed by the numerical experiments of
Section \ref{sec:Numerical}. Early stopping replaces the
finite-Fisher-information hypothesis by a finite second moment
(Corollary \ref{cor:early-stopped}), extending the guarantee, in Wasserstein
distance, to the singular data distributions that motivate diffusion models.
Finally, because every sampler in the field is a choice of forward and
sampler diffusion, the same accounting compares score-based diffusion, the
probability-flow ODE, flow matching, and stochastic interpolants within a
single inequality (Section \ref{sec:unified-comparison}), locating the
optimal sampler stochasticity at an explicit schedule functional and
recovering the empirical ODE/SDE trade-off.

\subsection{Outlooks and future work}\label{sec:future-work}

The present analysis is deliberately confined to the smooth, uniformly elliptic
setting, and several of its constants are stated up to dimension- and
schedule-dependent factors that we have not optimized. We record the main
limitations together with the extensions they suggest; each is self-contained
relative to the results above and is left for future work.
 
Explicit dimension dependence. The constants in
Theorem \ref{thm:total-error} and Corollary \ref{cor:early-stopped} are tracked
through norms of the drift and its derivatives and are not resolved into their
dependence on the ambient dimension $n$. Making this dependence explicit in
particular determining whether the marginal-KL bound is linear in $n$, as the
path-space analyses suggest would place the present estimates on the same
footing as the dimension-explicit convergence results and is a natural next
step. The numerical scaling we observe is consistent with a near-linear
dependence, but we do not claim a proof.
 
Empirical validation of the score and entropy-production terms.
The error decomposition predicts a quadratic dependence on the discretization
step, a quadratic dependence on the score error $\varepsilon^{2}$, and a
specific time profile of the entropy-production rate $e_p(t)$. Experiments
1--4 of Section \ref{sec:Numerical} probe the step-size, terminal-time, and
dimension scalings, and Experiments 5--6 are designed to probe the
$\varepsilon^{2}$ dependence and the measured $e_p(t)$ profile in
particular whether $e_p$ concentrates near the data end, which would
motivate non uniform step schedules from first principles. Completing those
measurements, and extending them to learned scores on non-Gaussian targets,
remains ongoing.
 
Singular and manifold-supported data. Corollary
\ref{cor:early-stopped} brings boundary-singular and manifold-supported data
within reach by early stopping, at the cost of a metric downgrade: the bound
controls the Kullback--Leibler divergence to the Gaussian-mollified target and
the Wasserstein distance to the true data. A direct treatment of the data
singularity---sharper than the $\mathcal O(\sqrt\delta)$ mollification gap, and
quantifying how the intrinsic dimension enters---would extend the framework to
the regime that motivates it most.
 
 Non-Gaussian endpoints and higher-order Fisher excess. The
endpoint optimality established in Section \ref{sec:error} is, at the level of
the early-stopping coefficient, a first-order statement: among noise kernels of
a given variance the Gaussian minimizes the Fisher information. The
discretization constant depends on higher spatial derivatives of the score,
which are controlled by the higher cumulants of the endpoint kernel; a
non-Gaussian endpoint contributes a higher-order Fisher excess at each order,
and the Gaussian is the unique kernel for which all such excesses vanish.
Tracking these higher-order constants explicitly, and the corresponding effect
on the reverse-drift derivatives and the log-Sobolev constant, is left for
future work.
 
Higher-order and deterministic discretizations. The
discretization analysis is carried out for the Euler--Maruyama scheme. The
unified bound of Section \ref{sec:unified-comparison} indicates that the
discretization constant depends on the sampler diffusion and is smaller for
deterministic samplers, which admit higher-order integrators; making the
discretization constant explicit for the probability-flow ODE and for
higher-order solvers, and comparing it across the diffusion levels of the
adjustable-noise family, would complete the ODE/SDE comparison quantitatively.
 
Discrete-state and geometric extensions. The mother identity
(Theorem \ref{thm:mother-identity}) uses only that the two marginal flows obey a
continuity equation, and an analogous identity holds for continuous-time Markov
jump processes through the master-equation relative-entropy dissipation, and for
diffusions on Riemannian manifolds or reflected on bounded domains. Carrying the
error analysis over to discrete-state diffusion models and to geometric and
boundary-constrained settings is a natural direction that the velocity-form
identity makes accessible.
 
\bibliographystyle{siamplain}

\bibliography{references}

\end{document}